\documentclass[10pt]{article}
\usepackage{amsthm}
\usepackage{amsmath}
\usepackage{amssymb}
\usepackage{makeidx}

\theoremstyle{definition}
\newtheorem{definition}{Definition}[section]

\newtheorem{example}[definition]{Example}

\newtheorem{remark}[definition]{Remark}

\theoremstyle{plain}

\newtheorem{lemma}[definition]{Lemma}
\newtheorem{proposition}[definition]{Proposition}
\newtheorem{corollary}[definition]{Corollary}

\numberwithin{equation}{section}

\def\N{{\mathbb N}}

\begin{document}
\title{Factor Congruence Lifting Property}
\author{George GEORGESCU and Claudia MURE\c SAN\thanks{Corresponding author.}\\ \footnotesize University of Bucharest\\ \footnotesize Faculty of Mathematics and Computer Science\\ \footnotesize Academiei 14, RO 010014, Bucharest, Romania\\ \footnotesize Emails: georgescu.capreni@yahoo.com; c.muresan@yahoo.com, cmuresan@fmi.unibuc.ro}
\date{\today }
\maketitle

\begin{abstract} In previous work, we have introduced and studied a lifting property in congruence--distributive universal algebras which we have defined based on the Boolean congruences of such algebras, and which we have called the Congruence Boolean Lifting Property. In a similar way, a lifting property based on factor congruences can be defined in congruence--distributive algebras; in this paper we introduce this property, which we have called the Factor Congruence Lifting Property, and study it, partly in relation to the Congruence Boolean Lifting Property, and to other lifting properties in particular classes of algebras.\\ {\em 2010 Mathematics Subject Classification:} Primary: 08B10; secondary: 03C05, 06F35, 03G25.\\ {\em Keywords:} Boolean Lifting Property; Boolean center; lattice; residuated lattice; reticulation; (congruence--distributive, congruence--permutable, arithmetical) algebra; factor congruence.\end{abstract}

\section{Introduction}
\label{introduction}

The Idempotent Lifting Property (abbreviated ILP or LIP), that is the property that every idempotent element can be lifted modulo every left (respectively right) ideal, is intensely studied in ring theory. The ILP is related to important classes of unitary rings: clean rings, exchange rings, Gelfand rings, maximal rings (\cite{banasch}, \cite{mcgov}, \cite{nic} etc.). In \cite{nic}, it is proven that any clean ring has ILP, and that the rings with ILP are exactly exchange rings; furthermore, in the commutative case, clean rings, exchange rings and rings with ILP coincide.

Lifting properties inspired by the ILP have been studied in algebras related to logic: MV--algebras \cite{figele}, BL--algebras \cite{din1}, \cite{leo}, (commutative) residuated lattices \cite{eu3}, \cite{eu}, \cite{ggcm}, \cite{dcggcm}, \cite{eudacs}, \cite{blpiasi}, bounded distributive lattices \cite{ggcm}, \cite{dcggcm}, \cite{blpiasi}. All these kinds of algebras have Boolean centers (subalgebras with a Boolean algebra structure), which allows the so--called Boolean Lifting Properties (BLP) to be defined. In bounded distributive lattices, three significant kinds of BLP naturally occur: the BLP modulo ideals (Id--BLP), the BLP modulo filters (Filt--BLP) and the BLP modulo all congruences (simply, BLP). In residuated lattices, a lifting property for idempotent elements (ILP) has also been studied \cite{dcggcm}. A generalization of these lifting properties to universal algebras, called the $\varphi $--Lifting Properties, have been studied in \cite{dcggcm}, \cite{eudacs}.

In the case of congruence--distributive universal algebras, a notion of Boolean Lifting Property can be defined, based on the Boolean center of the lattice of congruences of such an algebra; we have called it the Congruence Boolean Lifting Property \cite{cblp}. It turns out that the CBLP coincides to the BLP in residuated lattices (which includes BL--algebras and MV--algebras), but differs from the BLP, Id--BLP and Filt--BLP in the case of bounded distributive lattices, where CBLP is always present, unlike the BLP, Id--BLP and Filt--BLP. The study of universal algebras with CBLP is motivated by both their properties, including strong representation theorems and topological characterizations, and by the remarkable classes of universal algebras with CBLP, which include local algebras, discriminator equational classes etc..

As we have already mentioned, for defining the CBLP in a congruence--distributive algebra $A$, we have used ${\cal B}({\rm Con}(A))$, the Boolean center of the lattice of congruences of $A$. The present paper is concerned with the study of the Factor Congruence Lifting Property (FCLP); the FCLP is defined for congruence--distributive universal algebras, like the CBLP, except, instead of being defined starting from ${\cal B}({\rm Con}(A))$, the FCLP is defined based on the Boolean algebra ${\rm FC}(A)$ of the factor congruences of $A$; this Boolean algebra is present in a wider class of universal algebras than that of congruence--distributive algebras, thus the FCLP can be defined for this wider class; its study in this more general context remains a theme for future research; here we restrict our investigation to the context of congruence--distributive algebras, and compare the FCLP to the CBLP. These lifting properties coincide, for instance, in arithmetical algebras, but, in general, they differ, and, moreover, none of them implies the other; we shall see examples of finite non--distributive lattices with FCLP and without CBLP, and vice--versa. In residuated lattices, which are arithmetical algebras, the FCLP, CBLP and BLP coincide, while, in bounded distributive lattices, the FCLP coincides to the BLP, which implies that it differs from the CBLP, the Id--BLP and the Filt--BLP.

In Section \ref{preliminaries} of the present article, we recall some previously known notions and results from universal algebra and lattice theory that we use in the sequel. The results in the following sections are new, with the only exceptions of the results cited from other papers. In Section \ref{fclp}, we introduce the FCLP and obtain its main properties, including its preservation by quotients and finite direct products, a characterization for it through a certain behaviour of factor congruences in the lattice of congruences, and the fact that it coincides to the CBLP in arithmetical algebras. In Section \ref{vsblp}, we compare the FCLP to the CBLP and the BLP in residuated lattices and bounded distributive lattices, and prove that, in general, the CBLP does not imply the FCLP. In Section \ref{examples}, we provide some more properties of the FCLP, as well as many examples in lattices, in which we compare the FCLP to the CBLP; here we prove that the FCLP does not imply the CBLP either.

\section{Preliminaries}
\label{preliminaries}

For the purpose of self--containedness, in this section we present a set of results on the congruences of universal algebras, out of which most are well known, and the rest are straightforward. We refer the reader to \cite{bal}, \cite{blyth}, \cite{bur}, \cite{gralgu} for a further study of the notions and properties we recall here.

We shall denote by $\N $ the set of the natural numbers and by $\N ^*=\N \setminus \{0\}$. Throughout this paper, whenever there is no danger of confusion, any algebra shall be designated by its underlying set. All algebras shall be considerred non--empty, regardless of whether they have constants in their signature; by {\em trivial algebra} we mean an algebra with only one element, and by {\em non--trivial algebra} we mean an algebra with at least two distinct elements. All direct products and quotients of algebras shall be considerred with the operations defined canonically. For any non--empty family $(M_i)_{i\in I}$ of sets and any $\displaystyle M\subseteq \prod _{i\in I}M_i$, by $(x_i)_{i\in I}\in M$ we mean $x_i\in M_i$ for all $i\in I$, such that $(x_i)_{i\in I}\in M$. If we don`t specify otherwise, then we denote the (bounded) lattice operations, the Boolean operations and the partial orders in the usual way: $\vee ,\wedge ,\neg \, ,0,1,\leq $. For any lattice $L$, we denote by ${\rm Filt}(L)$ and ${\rm Id}(L)$ the set of the filters and that of the ideals of $L$, respectively; for any $X\subseteq L$, we shall denote by $[X)$ and $(X]$ the filter, respectively the ideal of $L$ generated by $X$; for any $x\in L$, we denote by $[x)$ the principal filter of $L$ generated by $x$: $[x)=[\{x\})=\{y\in L\ |\ x\leq y\}$. It is well known that bounded lattice morphisms between Boolean algebras are Boolean morphisms and surjective lattice morphisms between bounded lattices are bounded lattice morphisms; also, the congruences of any Boolean algebra coincide to the congruences of its underlying lattice.

Let $\tau $ be an arbitrary but fixed signature of universal algebras. Everywhere in this paper, except where it is mentioned otherwise, by {\em algebra} we shall mean $\tau $--algebra, by {\em morphism} we shall mean morphism of $\tau $--algebras, and {\em isomorphism} shall mean isomorphism of $\tau $--algebras. If $A$ and $B$ are two algebraic structures of the same kind and there is no danger of confusion, we shall denote by $A\cong B$ the fact that $A$ and $B$ are isomorphic.

Throughout the rest of this section, $A$ shall be an arbitrary algebra, unless mentioned otherwise. We shall denote by ${\rm Con}(A)$ the set of the congruences of $A$, by $\Delta _{A}=\{(a,a)\ |\ a\in A\}$ and by $\nabla _{A}=A^2$. Clearly, the algebra $A$ is non--trivial iff $\Delta _{A}\neq \nabla _{A}$. We shall denote by ${\rm Max}(A)$ the set of the {\em maximal congruences} of $A$, that is the maximal elements of $({\rm Con}(A)\setminus \{A\},\subseteq )$. $A$ is called a {\em local algebra} iff it has exactly one maximal congruence, and it is called a {\em semilocal algebra} iff it has only a finite number of maximal congruences. See in \cite{bur}, \cite{gralgu}, \cite{dcggcm}, \cite{cblp} the definition of a maximal algebra, and the property that all maximal algebras are semilocal algebras. Also, we shall denote by ${\rm Spec}(A)$ the set of the {\em prime congruences} of $A$, that is the congruences $\theta $ of $A$ which fulfill this condition: for all $\alpha ,\beta \in {\rm Con}(A)$, if $\alpha \cap \beta \subseteq \theta $, then $\alpha \subseteq \theta $ or $\beta \subseteq \theta $. For any $M\subseteq A^2$, we denote by $Cg_{A}(M)$ the congruence of $A$ generated by $M$; for any $a,b\in A$, the principal congruence $Cg_{A}(\{(a,b)\})$ shall also be denoted by $Cg_{A}(a,b)$. It is well known that $({\rm Con}(A),\vee ,\cap ,\Delta _{A},\nabla _{A})$ is a bounded lattice, where, for all $\phi ,\psi \in {\rm Con}(A)$, $\phi \vee \psi =Cg_{A}(\phi \cup \psi)$, and with $\subseteq $ as partial order; moreover, ${\rm Con}(A)$ is a complete lattice, in which, for any family $(\theta _i)_{i\in I}\subseteq {\rm Con}(A)$, $\displaystyle \bigvee _{i\in I}\theta _i=Cg_{A}(\bigcup _{i\in I}\theta _i)$. A congruence $\theta $ of $A$ is said to be {\em finitely generated} iff there exists a finite subset $X$ of $A^2$ such that $\theta =Cg_{A}(X)$; clearly, $\theta =Cg_{A}(\theta )$, so, if $\theta $ is finite, then it is finitely generated. The {\em radical} of $A$, denoted by ${\rm Rad}(A)$, is the intersection of the maximal congruences of $A$, which is a congruence of $A$ by the above. For any $\phi ,\psi \in {\rm Con}(A)$, we denote by $\phi \circ \psi $ the composition of $\phi $ with $\psi $: $\phi \circ \psi =\{(a,b)\in A^2\ |\ (\exists \, x\in A)\, ((a,x)\in \psi ,(x,b)\in \phi )\}$; note that $\phi \circ \psi $ is not always a congruence of $A$, and that $\phi \cup \psi \subseteq \phi \circ \psi $, because $\phi $ and $\psi $ are reflexive, that is $\phi \supseteq \Delta _{A}$ and $\psi \supseteq \Delta _{A}$, thus $\phi =\phi \circ \Delta _{A}\subseteq \phi \circ \psi $ and $\psi =\Delta _{A}\circ \psi \subseteq \phi \circ \psi $.

The algebra $A$ is said to be {\em congruence--distributive} iff the lattice ${\rm Con}(A)$ is distributive. $A$ is said to be {\em congruence--permutable} iff $\phi \circ \psi =\psi \circ \phi $ for all $\phi ,\psi \in {\rm Con}(A)$. $A$ is said to be {\em arithmetical} iff it is both congruence--distributive and congruence--permutable. For instance, it is well known that lattices are congruence--distributive algebras, and that Boolean algebras are arithmetical algebras.

Let $B$ be an algebra and $f:A\rightarrow B$ be a morphism. We denote by ${\rm Ker}(f)=\{(x,y)\in A^2\ |\ f(x)=f(y)\}$ the kernel of $f$. Clearly, ${\rm Ker}(f)\in {\rm Con}(A)$. For any $M\subseteq A^2$ and any $N\subseteq B^2$, we denote: $f(M)=\{(f(x),f(y))\ |\ (x,y)\in M\}$ and $f^{-1}(N)=\{(x,y)\in A^2\ |\ (f(x),f(y))\in N\}$. It is straightforward that, for any $\phi \in {\rm Con}(A)$ and any $\psi \in {\rm Con}(B)$, the following hold: $f^{-1}(\psi )\in {\rm Con}(A)$ and $f(f^{-1}(\psi ))=\psi $; if ${\rm Ker}(f)\subseteq \phi $, then $f(\phi )\in {\rm Con}(B)$ and $f^{-1}(f(\phi ))=\phi $. Therefore, if a $\theta \in {\rm Con}(A)$ has the property that ${\rm Ker}(f)\subseteq \theta $, then the mapping $\alpha \mapsto f(\alpha )$ is a bounded lattice isomorphism between the sublattice $[\theta )=\{\alpha \in {\rm Con}(A)\ |\ \theta \subseteq \alpha \}$ of ${\rm Con}(A)$ and ${\rm Con}(B)$, whose inverse maps $\beta \mapsto f^{-1}(\beta )$.

For any $\theta \in {\rm Con}(A)$, we shall denote by $A/\theta $ the quotient algebra of $A$ through $\theta $. Obviously, if $A/\theta $ is non--trivial, then so is $A$, thus, if $\Delta _{A/\theta }\neq \nabla _{A/\theta }$, then $\Delta _{A}\neq \nabla _{A}$. For any $a\in A$ and any $X\subseteq A$, we denote by $a/\theta $ the congruence class of $a$ with respect to $\theta $, and by $X/\theta =\{x/\theta \ |\ x\in X\}$. We shall denote by $p_{\theta }:A\rightarrow A/\theta $ the canonical surjective morphism: $p_{\theta }(a)=a/\theta $ for all $a\in A$. We also denote, for any $M\subseteq A^2$, by $M/\theta =p_{\theta }(M)=\{(a/\theta ,b/\theta )\ |\ (a,b)\in M\}$. Clearly, ${\rm Ker}(p_{\theta })=\theta $, hence, by the above, the mapping $\alpha \mapsto p_{\theta }(\alpha )=\alpha /\theta $ is a bounded lattice isomorphism between $[\theta )$ and ${\rm Con}(A/\theta )$, whose inverse maps $\beta \mapsto p_{\theta }^{-1}(\beta )$. We shall denote by $s_{\theta }:{\rm Con}(A/\theta )\rightarrow [\theta )$ the bounded lattice isomorphism defined by: 
$s_{\theta }(\beta )=p_{\theta }^{-1}(\beta )=\{(a,b)\in A^2\ |\ (a/\theta ,b/\theta )\in \beta \}$ for all $\beta \in {\rm Con}(A/\theta )$.

Now let $\theta \in {\rm Con}(A)$ and $\alpha ,\beta \in [\theta )$, arbitrary. Since $s_{\theta }^{-1}:[\theta )\rightarrow {\rm Con}(A/\theta )$, $s_{\theta }^{-1}(\gamma )=p_{\theta }(\gamma )=\gamma /\theta $ for all $\gamma \in [\theta )$, is a lattice isomorphism, and thus it is injective, we have: $\alpha /\theta =\beta /\theta $ iff $\alpha =\beta $. Notice, moreover, that, for any $a,b\in A$, the following equivalence holds: $(a/\theta ,b/\theta )\in \alpha /\theta $ iff $(a,b)\in \alpha $. Indeed, $(a,b)\in \alpha $ implies $(a/\theta ,b/\theta )\in \alpha /\theta $ by the very definition of $\alpha /\theta $; conversely, if $(a/\theta ,b/\theta )\in \alpha /\theta $, then there exist $a^{\prime },b^{\prime }\in A$ such that $a^{\prime }/\theta =a/\theta $, $b^{\prime }/\theta =b/\theta $ and $(a^{\prime },b^{\prime })\in \alpha $, which means that $(a,a^{\prime })\in \theta \subseteq \alpha $, $(a^{\prime },b^{\prime })\in \alpha $ and $(b^{\prime },b)\in \theta \subseteq \alpha $, hence $(a,b)\in \alpha $ by the transitivity of $\alpha $.

Now let us notice that $\alpha \circ \beta \in [\theta )$ and $(\alpha \circ \beta )/\theta =\alpha /\theta \circ \beta /\theta $. First, since $\beta $ is reflexive, it follows that $\alpha \circ \beta \supseteq \alpha \circ \Delta _{A}=\alpha \supseteq \theta $, so $\alpha \circ \beta \in [\theta )$. Now let $a,b\in A$. If $(a/\theta ,b/\theta )\in (\alpha \circ \beta )/\theta $, then there exist $a^{\prime },b^{\prime }\in A$ such that $a^{\prime }/\theta =a/\theta $, $b^{\prime }/\theta =b/\theta $ and $(a^{\prime },b^{\prime })\in \alpha \circ \beta $, which means that there exists an $x\in A$ such that $(a^{\prime },x)\in \beta $ and $(x,b^{\prime })\in \alpha $; then $(a^{\prime }/\theta ,x/\theta )=(a/\theta ,x/\theta )\in \beta /\theta $ and $(x/\theta ,b^{\prime }/\theta )=(x/\theta ,b/\theta )\in \alpha /\theta $, thus $(a/\theta ,b/\theta )\in \alpha /\theta \circ \beta /\theta $. Conversely, if $(a/\theta ,b/\theta )\in \alpha /\theta \circ \beta /\theta $, then there exists an $x\in A$ such that $(a/\theta ,x/\theta )\in \beta /\theta $ and $(x/\theta ,b/\theta )\in \alpha /\theta $; by the above, this is equivalent to $(a,x)\in \beta $ and $(x,b)\in \alpha $, which implies $(a,b)\in \alpha \circ \beta $, thus $(a/\theta ,b/\theta )\in (\alpha \circ \beta )/\theta $. So indeed $(\alpha \circ \beta )/\theta =\alpha /\theta \circ \beta /\theta $, thus, since $\alpha $ and $\beta $ are arbitrary: $\alpha /\theta \circ \beta /\theta =\beta /\theta \circ \alpha /\theta $ iff $(\alpha \circ \beta )/\theta =(\beta \circ \alpha )/\theta $ iff $\alpha \circ \beta =\beta \circ \alpha $ by the above. Since $s_{\theta }^{-1}$ is surjective, that is ${\rm Con}(A/\theta )=s_{\theta }^{-1}({\rm Con}(A))=\{\gamma /\theta \ |\ \gamma \in [\theta )\}$, it follows that: $A/\theta $ is congruence--permutable iff the congruences in $[\theta )$ permute with respect to composition; in particular, if $A$ is congruence--permutable, then $A/\theta $ is congruence--permutable.

Throughout the rest of this section, the algebra $A$ shall be congruence--distributive and $\theta \in {\rm Con}(A)$. Since ${\rm Con}(A)$ is distributive, it follows that $[\theta )$ is distributive, hence ${\rm Con}(A/\theta )$ is distributive since ${\rm Con}(A/\theta )\cong [\theta )$.

Let us note, from the above, that all the quotient algebras of a congruence--distributive algebra are congruence--distributive, and all the quotient algebras of a congruence--permutable algebra are congruence--permutable. Consequently, all the quotient algebras of an arithmetical algebra are arithmetical.

\begin{lemma}{\rm \cite[Theorem $2.3$, (iii)]{bj}} For any $M\subseteq A^2$, $Cg_{A/\theta }(M/\theta )=(Cg_{A}(M)\vee \theta )/\theta $.\label{l2.1}\end{lemma}

Now let us consider the functions: $u_{\theta }:{\rm Con}(A)\rightarrow {\rm Con}(A/\theta )$ and $v_{\theta }:{\rm Con}(A)\rightarrow [\theta )$, defined by: $u_{\theta }(\alpha )=(\alpha \vee \theta )/\theta $ and $v_{\theta }(\alpha )=\alpha \vee \theta $ for all $\alpha \in {\rm Con}(A)$. Then, clearly, $u_{\theta }$ and $v_{\theta }$ are bounded lattice morphisms, and the following diagram is commutative (see also \cite{cblp}):\vspace*{-10pt}

\begin{center}\begin{picture}(120,53)(0,0)
\put(0,30){${\rm Con}(A)$}
\put(35,33){\vector (1,0){40}}
\put(15,27){\vector (3,-2){33}}
\put(94,26){\vector (-3,-2){33}}
\put(50,38){$u_{\theta }$}
\put(23,11){$v_{\theta }$}
\put(80,10){$s_{\theta }$}
\put(76,30){${\rm Con}(A/\theta )$}
\put(50,0){$[\theta )$}
\end{picture}\end{center}

Throughout the following sections, we shall keep the notations for the surjective morphism $p_{\theta }$, the bounded lattice morphisms $u_{\theta },v_{\theta }$ and the bounded lattice isomorphism $s_{\theta }$, for any congruence--distributive algebra $A$ and any $\theta \in {\rm Con}(A)$. In the same context, we shall denote by $\neg _{\theta }$ the complementation in the Boolean algebra ${\cal B}([\theta ))$. 

\begin{remark} For all $\alpha \in [\theta )$, $v_{\theta }(\alpha )=\alpha \vee \theta =\alpha $, thus $v_{\theta }$ is surjective. Since $v_{\theta }=s_{\theta }\circ u_{\theta }$ and $s_{\theta }$ is bijective, it follows that $u_{\theta }$ is surjective, as well.\label{4***}\end{remark}

\begin{remark} Let $n\in \N ^*$ and $A_1,\ldots ,A_n$ be algebras. If $\theta _i\in {\rm Con}(A_i)$ for all $i\in \overline{1,n}$, then we denote by $\theta _1\times \ldots \times \theta _n=\{((x_1,\ldots ,x_n),(y_1,\ldots ,y_n))\ |\ (\forall \, i\in \overline{1,n})\, ((x_i,y_i)\in \theta _i)\}$; it is immediate that $\displaystyle \theta _1\times \ldots \times \theta _n\in {\rm Con}(\prod_{i=1}^nA_i)$. Furthermore, according to \cite[Corollary $4.2$]{isk}, the mapping $(\theta _1,\ldots ,\theta _n)\mapsto \theta _1\times \ldots \times \theta _n$ is a bounded lattice isomorphism from $\displaystyle \prod_{i=1}^n{\rm Con}(A_i)$ to $\displaystyle {\rm Con}(\prod_{i=1}^nA_i)$ which preserves $\circ $, that is, if $\alpha _i,\beta _i\in {\rm Con}(A_i)$ for all $i\in \overline{1,n}$, then $(\alpha _1\times \ldots \times \alpha _n)\circ (\beta _1\times \ldots \times \beta _n)=(\alpha _1\circ \beta _1)\times \ldots \times (\alpha _n\circ \beta _n)$. Consequently, if $A_1,\ldots ,A_n$ are congruence--distributive (respectively congruence--permutable, respectively arithmetical), then so is $\displaystyle \prod_{i=1}^nA_i$.

Now let $\displaystyle A=\prod_{i=1}^nA_i$. Then $\displaystyle A/\theta \cong \prod_{i=1}^nA_i/\theta _i$. Indeed, let us define $\displaystyle j:\prod_{i=1}^nA_i/\theta _i\rightarrow A/\theta $ by: for all $a_1\in A_1,\ldots ,a_n\in A_n$, $j(a_1/\theta _1,\ldots ,a_n/\theta _n)=(a_1,\ldots ,a_n)/\theta $. Then, clearly, $j$ is surjective. Now let $a_1,b_1\in A_1,\ldots ,a_n,b_n\in A_n$, such that $j(a_1/\theta _1,\ldots ,a_n/\theta _n)=j(b_1/\theta _1,\ldots ,b_n/\theta _n)$, that is $(a_1,\ldots ,a_n)/\theta =(b_1,\ldots ,b_n)/\theta $, so that $((a_1,\ldots ,a_n),(b_1,\ldots ,b_n))\in \theta =\theta _1\times \ldots \times \theta _n$, which means that, for all $i\in \overline{1,n}$, $(a_i,b_i)\in \theta _i$, that is, for all $i\in \overline{1,n}$, $(a_i,b_i)\in \theta _i$, $a_i/\theta _i=b_i/\theta _i$, so $(a_1/\theta _1,\ldots ,a_n/\theta _n)=(b_1/\theta _1,\ldots ,b_n/\theta _n)$. Hence $j$ is injective. It is straightforward that $j$ is a morphism. Therefore $j$ is an isomorphism.\label{oremarca}\end{remark}

So let us note that finite direct products of congruence--distributive algebras are congruence--distributive, and finite direct products of congruence--permutable algebras are congruence--permutable, hence finite direct products of arithmetical algebras are arithmetical.

For any bounded distributive lattice $L$, we denote by ${\cal B}(L)$ the {\em Boolean center} of $L$, that is the Boolean sublattice of $L$ made of the complemented elements of $L$, which, obviously, is the largest Boolean sublattice of $L$. If $M$ is also a bounded distributive lattice and $f:L\rightarrow M$ is a bounded lattice morphism, then $f({\cal B}(L))\subseteq {\cal B}(M)$, thus we can define ${\cal B}(f)=f\mid _{{\cal B}(L)}:{\cal B}(L)\rightarrow {\cal B}(M)$, which is a bounded lattice morphism between two Boolean algebras, and thus it is a Boolean morphism. In this way, ${\cal B}$ becomes a covariant functor from the category of bounded distributive lattices to the category of Boolean algebras.

We shall call the congruences from ${\cal B}({\rm Con}(A))$ the {\em Boolean congruences} of $A$. A congruence $\phi $ of $A$ is called a {\em factor congruence} iff there exists a congruence $\phi ^*$ of $A$ such that $\phi \vee \phi ^*=\nabla _{A}$, $\phi \cap \phi ^*=\Delta _{A}$ and $\phi \circ \phi ^*=\phi ^*\circ \phi $; in this case, $(\phi ,\phi ^*)$ is called a {\em pair of factor congruences}. We denote by ${\rm FC}(A)$ the set of the factor congruences of $A$. Clearly, if $(\phi ,\phi ^*)$ is a pair of factor congruences, then $\phi ^*\in {\rm FC}(A)$ and it is uniquely determined by $\phi ^*=\neg \, \phi $, and hence ${\rm FC}(A)=\{\phi \in {\cal B}({\rm Con}(A))\ |\ \phi \circ \neg \, \phi =\neg \, \phi \circ \phi \}$. In other words, the factor congruences of $A$ are the Boolean congruences of $A$ which permute with their complement with respect to composition. Thus, if the algebra $A$ is arithmetical, then ${\rm FC}(A)={\cal B}({\rm Con}(A))$. Clearly, $(\Delta _{A},\nabla _{A})$ is a pair of factor congruences of $A$. Moreover, according to \cite{cjt}, ${\rm FC}(A)$ is a Boolean sublattice of ${\rm Con}(A)$, and thus a Boolean subalgebra of ${\cal B}({\rm Con}(A))$ (see also \cite{isk}). Consequently, if ${\cal B}({\rm Con}(A))=\{\Delta _{A},\nabla _{A}\}$, then ${\rm FC}(A)={\cal B}({\rm Con}(A))=\{\Delta _{A},\nabla _{A}\}$; also, if ${\rm Con}(A)=\{\Delta _{A},\nabla _{A}\}$, then ${\rm FC}(A)={\cal B}({\rm Con}(A))={\rm Con}(A)=\{\Delta _{A},\nabla _{A}\}$.

Let $S$ be an arbitrary set. We shall denote by ${\rm Eq}(S)$ the set of the equivalences on $S$ and, for any $\rho \subseteq S^2$, by $\rho ^{-1}=\{(y,x)\in S^2\ |\ (x,y)\in \rho \}$ and by $\rho ^2=\rho \circ \rho $. So ${\rm Eq}(S)=\{\rho \ |\ \rho \subseteq S^2,\rho \supseteq \Delta _S,\rho =\rho ^{-1},\rho ^2\subseteq \rho \}$. Now let $\rho ,\sigma \in {\rm Eq}(S)$. Then: $\rho \circ \sigma \in {\rm Eq}(S)$ iff $\rho \circ \sigma =\sigma \circ \rho $ iff $\rho \circ \sigma =(\rho \circ \sigma )^{-1}$. Indeed, we always have: $(\rho \circ \sigma )^{-1}=\sigma ^{-1}\circ \rho ^{-1}=\sigma \circ \rho $, hence the last of the equivalences above, and $\rho \circ \sigma \supseteq \Delta _S\circ \Delta _S=\Delta _S$; also, if $\rho \circ \sigma \in {\rm Eq}(S)$, then $\rho \circ \sigma =(\rho \circ \sigma )^{-1}$; conversely, if $\rho \circ \sigma =(\rho \circ \sigma )^{-1}$, then, by the above: $\rho \circ \sigma =\sigma \circ \rho $, thus $(\rho \circ \sigma )^2=\rho \circ \sigma \circ \rho \circ \sigma =\rho \circ \rho \circ \sigma \circ \sigma =\rho ^2\circ \sigma ^2\subseteq \rho \circ \sigma $, therefore $\rho \circ \sigma \in {\rm Eq}(S)$.

Now let $\phi ,\psi \in {\rm Con}(A)$, arbitrary. Then, clearly, $\phi \circ \psi $ preserves the operations of $A$, hence, by the above: $\phi \circ \psi \in {\rm Con}(A)$ iff $\phi \circ \psi \in {\rm Eq}(A)$ iff $\phi \circ \psi =\psi \circ \phi $ iff $\phi \circ \psi =(\phi \circ \psi )^{-1}$. Let us notice that $\phi \cup \psi \subseteq \phi \circ \psi \subseteq \phi \vee \psi $. Indeed, we have already seen that $\phi \cup \psi \subseteq \phi \circ \psi $; now let $(a,b)\in \phi \circ \psi $, so that $(a,x)\in \psi $ and $(x,b)\in \phi $ for some $x\in A$; since $\phi \subseteq \phi \vee \psi $ and $\psi \subseteq \phi \vee \psi $, it follows that $(a,x),(x,b)\in \phi \vee \psi $, hence $(a,b)\in \phi \vee \psi $ by the transitivity of the congruence $\phi \vee \psi $. Thus, if $\phi \circ \psi =\nabla _{A}$, then $\phi \circ \psi =\phi \vee \psi =\nabla _{A}$; also, if $\phi \cup \psi =\nabla _{A}$, then $\phi \cup \psi =\phi \circ \psi =\phi \vee \psi =\nabla _{A}$. Furthermore, $\phi \circ \psi \in {\rm Con}(A)$ iff $\phi \circ \psi =\phi \vee \psi $ iff $\phi \vee \psi \subseteq \phi \circ \psi $, where the second equivalence is obvious from the above and the converse implication in the first equivalence is trivial, and, since $\phi \cup \psi \subseteq \phi \circ \psi $ and $\phi \vee \psi =Cg_{A}(\phi \cup \psi )$, it follows that: $\phi \circ \psi \in {\rm Con}(A)$ implies $\phi \vee \psi \subseteq \phi \circ \psi $. Consequently, if $\phi \vee \psi =\nabla _{A}$, then: $\phi \circ \psi \in {\rm Con}(A)$ iff $\phi \circ \psi =\nabla _{A}$. From the above, it follows that, for any $\phi \in {\cal B}({\rm Con}(A))$, the following equivalences hold: $\phi \in {\rm FC}(A)$ iff $\phi \circ \neg \, \phi \in {\rm Con}(A)$ iff $\phi \circ \neg \, \phi =\neg \, \phi \circ \phi $ iff $\phi \circ \neg \, \phi =(\phi \circ \neg \, \phi )^{-1}$ iff $\phi \circ \neg \, \phi =\nabla _{A}$ iff $\neg \, \phi \in {\rm FC}(A)$. 

\begin{remark} Let $\phi ,\phi ^*,\psi \in {\rm Con}(A)$. Then:\begin{enumerate}
\item\label{caractfc1} $(\phi ,\phi ^*)$ is a pair of factor congruences iff $\phi \circ \phi ^*=\nabla _{A}$ and $\phi \cap \phi ^*=\Delta _{A}$ (\cite{bj});
\item\label{caractfc2} if $\phi \in {\rm FC}(A)$, then $\phi \circ \psi =\psi \circ \phi $ (\cite[Theorem $3$]{isk} and \cite[Theorem $3$]{ssm}).\end{enumerate}\label{caractfc}\end{remark}

\begin{remark}\begin{enumerate}
\item\label{fc1} Let $\phi \in {\rm FC}(A)$ and $\psi \in {\rm Con}(A)$. Then $\phi \circ \psi =\phi \vee \psi $. Indeed, by Remark \ref{caractfc}, (\ref{caractfc2}), we have $\phi \circ \psi =\psi \circ \phi $, which implies
$\phi \circ \psi =\phi \vee \psi $ by the above.
\item\label{fc2} If $A$ is an arithmetical algebra, then all $\phi ,\psi \in {\rm Con}(A)$ fulfill $\phi \circ \psi =\psi \circ \phi $, thus they all fulfill $\phi \circ \psi =\phi \vee \psi $ by the above.\end{enumerate}\label{fc}\end{remark}

\begin{remark} Let $B$ be a congruence--distributive algebra such that there exists an isomorphism $f:A\rightarrow B$. Then it is straightforward that the mapping $\theta \mapsto f(\theta )$ is a bounded lattice isomorphism between ${\rm Con}(A)$ and ${\rm Con}(B)$ and a Boolean isomorphism between ${\cal B}({\rm Con}(A))$ and ${\cal B}({\rm Con}(B))$, as well as between ${\rm FC}(A)$ and ${\rm FC}(B)$. If we replace $A$ and $B$ by two lattices $L$ and $M$, respectively, then the above also hold if $f:L\rightarrow M$ is a dual lattice isomorphism.\label{transpcong}\end{remark}

\begin{lemma}{\rm \cite{ssm}} Let $L$ be a bounded distributive lattice. Then the function $f_L:{\cal B}(L)\rightarrow {\rm FC}(L)$, defined by $f_L(a)=Cg_L(a,0)$ for all $a\in {\cal B}(L)$, is a Boolean isomorphism.\label{lemma3.2}\end{lemma}

\begin{lemma} Let $n\in \N ^*$ and $A_1,\ldots ,A_n$ be congruence--distributive algebras. Then the mapping $(\theta _1,\ldots ,\theta _n)\mapsto \theta _1\times \ldots \times \theta _n$ sets a Boolean isomorphism between the Boolean algebras $\displaystyle \prod _{i=1}^n{\cal B}({\rm Con}(A_i))$ and $\displaystyle {\cal B}({\rm Con}(\prod _{i=1}^nA_i))$, as well as between the Boolean algebras $\displaystyle \prod _{i=1}^n{\rm FC}(A_i)$ and $\displaystyle {\rm FC}(\prod _{i=1}^nA_i)$.\label{lemma3.3}\end{lemma}

\begin{proof} Notice that $\displaystyle \prod _{i=1}^n{\cal B}({\rm Con}(A_i))={\cal B}(\prod _{i=1}^n{\rm Con}(A_i))$, thus the mapping $(\theta _1,\ldots ,\theta _n)\mapsto \theta _1\times \ldots \times \theta _n$ between $\displaystyle \prod _{i=1}^n{\cal B}({\rm Con}(A_i))$ and $\displaystyle {\cal B}({\rm Con}(\prod _{i=1}^nA_i))$ is well defined and it is a Boolean isomorphism, namely the image through the functor ${\cal B}$ of the bounded lattice morphism from Remark \ref{oremarca}. The statement on $\displaystyle \prod _{i=1}^n{\rm FC}(A_i)$ and $\displaystyle {\rm FC}(\prod _{i=1}^nA_i)$ follows from \cite[Theorem $11$]{isk}, or straightforward from Remark \ref{oremarca}.\end{proof}

Let $\Omega \subseteq {\rm Con}(A)$. We say that $\Omega $ satisfies the {\em Chinese Remainder Theorem} ({\em CRT}, for short) iff, for all $n\in \N ^*$, all $\theta _1,\ldots ,\theta _n\in \Omega $ and all $a_1,\ldots ,a_n\in A$ such that $(a_i,a_j)\in \theta _i\vee \theta _j$ for all $i,j\in \overline{1,n}$, there exists an $a\in A$ such that $(a,a_i)\in \theta _i$ for all $i\in \overline{1,n}$. We say that $A$ satisfies the {\em CRT} iff ${\rm Con}(A)$ satisfies the CRT.

\begin{proposition}{\rm \cite{cor}} Let $\Omega $ be a bounded sublattice of ${\rm Con}(A)$. Then $\Omega $ fulfills the CRT iff the bounded lattice $\Omega $ is distributive and all $\alpha ,\beta \in \Omega $ satisfy $\alpha \circ \beta =\beta \circ \alpha $.\label{proposition3.4}\end{proposition}

\begin{corollary}\begin{enumerate}
\item\label{corollary3.5(1)} If $A$ is congruence--distributive, then ${\rm FC}(A)$ fulfills the CRT.
\item\label{corollary3.5(2)} $A$ fulfills the CRT iff $A$ is arithmetical.\end{enumerate}\label{corollary3.5}\end{corollary}

\section{FCLP: Definition, Main Properties, Characterization}
\label{fclp}

In this section we provide some more results on factor congruences, introduce the Factor Congruence Lifting Property, and obtain some of its properties, including its preservation by quotients and finite direct products, and a characterization for it through a certain property of the lattice of congruences that we have called FC--normality. We also recall the Congruence Boolean Lifting Property from \cite{cblp} and start comparing these two lifting properties; we will show more on the way they relate to each other in the following sections.

\begin{proposition} Let $n\in \N ^*$ and $A,A_1,\ldots ,A_n$ be congruence--distributive algebras. Then the following statements are equivalent:\begin{enumerate}
\item\label{proposition3.6(1)} $\displaystyle A\cong \prod _{i=1}^nA_i$;
\item\label{proposition3.6(2)} there exist $\alpha _1,\ldots ,\alpha _n\in {\rm FC}(A)$ such that $\displaystyle \bigcap _{i=1}^n\alpha _i=\Delta _{A}$, $\alpha _i\vee \alpha _j=\nabla _{A}$ for all $i,j\in \overline{1,n}$ with $i\neq j$, and $A_i\cong A/\alpha _i$ for all $i\in \overline{1,n}$.\end{enumerate}\label{proposition3.6}\end{proposition}

\begin{proof} (\ref{proposition3.6(1)})$\Rightarrow $(\ref{proposition3.6(2)}): If $n=1$, then just take $\alpha _1=\Delta _{A}\in {\rm FC}(A)$. Now assume that $n\geq 2$. Clearly, we may assume that $\displaystyle A=\prod _{i=1}^nA_i$. For each $i\in \overline{1,n}$, let $\pi _i:A\rightarrow A_i$ be the canonical projection: for all $(a_1,\ldots ,a_n)\in A$, $\pi _i(a_1,\ldots ,a_n)=a_i$, and
$\alpha _i={\rm Ker}(\pi _i)=\{(a,b)\in A^2\ |\ \pi _i(a)=\pi _i(b)\}=\{((a_1,\ldots ,a_n),(b_1,\ldots ,b_n))\in A^2\ |\ a_i=b_i\}\in {\rm Con}(A_i)$, since $\pi _i$ is a morphism. Clearly, $\displaystyle \bigcap _{i=1}^n\alpha _i=\Delta _{A}$.

Let $i,j\in \overline{1,n}$ such that $i\neq j$, and let $a=(a_1,\ldots ,a_n),b=(b_1,\ldots ,b_n)\in A$, arbitrary. Let $x=(a_1,\ldots ,a_{i-1},b_i,a_{i+1},\ldots ,a_n)\in A$. Since $i\neq j$, we have $\pi _i(a)=a_i=\pi _i(x)$, that is $(a,x)\in \alpha _i$. We also have $\pi _i(x)=b_i=\pi _i(b)$, that is $(x,b)\in \alpha _i$. Thus $(a,x),(x,b)\in \alpha _i\vee \alpha _j$, so $(a,b)\in \alpha _i\vee \alpha _j$ by the transitivity of $\alpha _i\vee \alpha _j$. Hence $\alpha _i\vee \alpha _j=A^2=\nabla _{A}$.

Now let $i\in \overline{1,n}$. $A/\alpha _i=\{a/\alpha _i\ |\ a\in A\}$, where, for all $a\in A$, $a/\alpha _i=\{b\in A\ |\ (a,b)\in \alpha _i\}=\{b\in A\ |\ \pi _i(a)=\pi _i(b)\}$ Let $f_i:A/\alpha _i\rightarrow A_i$, for all $a\in A$, $f_i(a/\alpha _i)=\pi _i(a)$. Then, clearly, $f_i$ is well defined and it is an isomorphism, thus $A_i\cong A/\alpha _i$.

Let $\displaystyle \beta _i=\bigcap _{j\in \overline{1,n}\setminus \{i\}}\alpha _j\in {\rm Con}(A)$. Then $\displaystyle \alpha _i\cap \beta _i=\bigcap _{j\in \overline{1,n}}\alpha _j=\Delta _{A}$. Now let $a=(a_1,\ldots ,a_n),b=(b_1,\ldots ,b_n)\in A$, arbitrary, and let $x=(a_1,\ldots ,a_{i-1},b_i,a_{i+1},\ldots ,a_n)\in A$. For all $j\in \overline{1,n}\setminus \{i\}$, $\pi _j(a)=a_j=\pi _j(x)$, which means that $\displaystyle (a,x)\in \bigcap _{j\in \overline{1,n}\setminus \{i\}}\alpha _j=\beta _i$; and $\pi _i(x)=b_i=\pi _i(b)$, that is $(x,b)\in \alpha _i$. Thus $(a,b)\in \alpha _i\circ \beta _i$, hence $\alpha _i\circ \beta _i=\nabla _{A}$. By Remark \ref{caractfc}, (\ref{caractfc1}), it follows that $(\alpha _i,\beta _i)$ is a pair of factor congruences. 

Therefore $\alpha _1,\ldots ,\alpha _n\in {\rm FC}(A)$.

\noindent (\ref{proposition3.6(2)})$\Rightarrow $(\ref{proposition3.6(1)}): Let us define $\displaystyle f:A\rightarrow \prod _{i=1}^nA/\alpha _i$ by: $f(a)=(a/\alpha _1,\ldots ,a/\alpha _n)$ for all $a\in A$. Let $a,b\in A$ such that $f(a)=f(b)$, that is $a/\alpha _i=b/\alpha _i$ for all $i\in \overline{1,n}$, which means that $\displaystyle (a,b)\in \bigcap _{i=1}^n\alpha _i=\Delta _{A}$, that is $a=b$, so $f$ is injective. For all $i\in \overline{1,n}$, let $q_i\in A/\alpha _i$, so that $q_i=a_i/\alpha _i$ for some $a_i\in A$. Then, for all $i,j\in \overline{1,n}$ with $i\neq j$, $(a_i,a_j)\in A^2=\nabla _{A}=\alpha _i\vee \alpha _j$. By Corollary \ref{corollary3.5}, (\ref{corollary3.5(1)}), ${\rm FC}(A)$ satisfies the CRT, hence there exists an $a\in A$ with the property that, for all $i\in \overline{1,n}$, $(a,a_i)\in \alpha _i$, that is $a/\alpha _i=a_i/\alpha _i$. Therefore $f(a)=(a/\alpha _1,\ldots ,a/\alpha _n)=(a_1/\alpha _1,\ldots ,a_n/\alpha _n)=(q_1,\ldots ,q_n)$; thus $f$ is surjective. Clearly, $f$ is a morphism. Therefore $f$ is an isomorphism, so $\displaystyle A\cong \prod _{i=1}^nA/\alpha _i\cong \prod _{i=1}^nA_i$.\end{proof}

Throughout the rest of this section, $A$ shall be a congruence--distributive algebra and $\theta \in {\rm Con}(A)$, arbitrary.

$v_{\theta }:{\rm Con}(A)\rightarrow [\theta )$ is a bounded lattice morphism, thus ${\cal B}(v_{\theta }):{\cal B}({\rm Con}(A))\rightarrow {\cal B}([\theta ))$ is a Boolean morphism, hence, for all $\psi \in {\cal B}({\rm Con}(A))$, we have: $\psi \vee \theta =v_{\theta }(\psi )={\cal B}(v_{\theta })(\psi )\in {\cal B}([\theta ))$, and $\neg _{\theta }(\psi \vee \theta )=\neg _{\theta }({\cal B}(v_{\theta })(\psi ))={\cal B}(v_{\theta })(\neg \, \psi )=v_{\theta }(\neg \, \psi )=\neg \, \psi \vee \theta $. $s_{\theta }^{-1}:[\theta )\rightarrow {\rm Con}(A/\theta )$ is a bounded lattice isomorphism, thus ${\cal B}(s_{\theta }^{-1}):{\cal B}([\theta ))\rightarrow {\cal B}({\rm Con}(A/\theta ))$ is a Boolean isomorphism, hence ${\cal B}({\rm Con}(A/\theta ))={\cal B}(s_{\theta }^{-1})({\cal B}([\theta )))=s_{\theta }^{-1}({\cal B}([\theta ))=\{\psi /\theta \ |\ \psi \in {\cal B}([\theta ))\}$ and, for any $\psi \in {\cal B}([\theta ))$, the complement of $\psi /\theta $ in ${\cal B}({\rm Con}(A/\theta ))$ is $\neg \, (\psi /\theta )=(\neg _{\theta }\psi )/\theta $. By the above, for any $\psi \in {\cal B}(A)$, it follows that $(\psi \vee \theta )/\theta \in {\cal B}({\rm Con}(A/\theta ))$ and $\neg \, ((\psi \vee \theta )/\theta )=(\neg _{\theta }(\psi \vee \theta ))/\theta =(\neg \, \psi \vee \theta )/\theta $. Therefore ${\rm FC}(A/\theta )=\{\gamma \in {\cal B}({\rm Con}(A/\theta ))\ |\ \gamma \circ \neg \, \gamma =\neg \, \gamma \circ \gamma \}=\{\psi /\theta \ |\ \psi \in {\cal B}([\theta )),\psi /\theta \circ (\neg _{\theta }\psi )/\theta =(\neg _{\theta }\psi )/\theta \circ \psi /\theta \}=\{\psi /\theta \ |\ \psi \in {\cal B}([\theta )),\psi \circ \neg _{\theta }\psi =\neg _{\theta }\psi \circ \psi \}=\{\psi /\theta \ |\ \psi \in {\rm FC}([\theta ))\}=p_{\theta }({\rm FC}([\theta )))={\rm FC}([\theta ))/\theta $.

\begin{proposition} $u_{\theta }({\rm FC}(A))\subseteq {\rm FC}(A/\theta )$.\label{proposition3.7}\end{proposition}

\begin{proof} Let $\psi \in {\rm FC}(A)$, which means that $\psi \in {\cal B}({\rm Con}(A))$ and $\psi \circ \neg \, \psi =\nabla _{A}$. Then, by the above, $u_{\theta }(\psi )=(\psi \vee \theta )/\theta \in {\cal B}({\rm Con}(A/\theta ))$ and $\neg _{\theta }u_{\theta }(\psi )=\neg \, ((\psi \vee \theta )/\theta )=(\neg \, \psi \vee \theta )/\theta $; also, $(\psi \vee \theta )\circ (\neg \, \psi \vee \theta )\supseteq \psi \circ \neg \, \psi =\nabla _{A}$, thus $u_{\theta }(\psi )\circ \neg \, u_{\theta }(\psi )=(\psi \vee \theta )/\theta \circ (\neg \, \psi \vee \theta )/\theta =((\psi \vee \theta )\circ (\neg \, \psi \vee \theta ))/\theta =\nabla _{A}/\theta =\nabla _{A/\theta }$. Therefore $u_{\theta }(\psi )\in {\rm FC}(A/\theta )$.\end{proof}

We denote by ${\rm FC}(\theta )=u_{\theta }\mid _{{\rm FC}(A)}:{\rm FC}(A)\rightarrow {\rm FC}(A/\theta )$.

\begin{proposition}\begin{enumerate}
\item\label{proposition3.8(1)} ${\rm FC}(\theta )$ is well defined and it is a Boolean morphism;
\item\label{proposition3.8(2)} the following diagrams are commutative:\vspace*{-23pt}

\begin{center}\begin{picture}(200,70)(0,0)
\put(0,45){${\rm FC}(A)$}
\put(70,45){${\cal B}({\rm Con}(A))$} 
\put(155,45){${\rm Con}(A)$}
\put(0,5){${\rm FC}(A/\theta )$}
\put(70,5){${\cal B}({\rm Con}(A/\theta ))$} 
\put(155,5){${\rm Con}(A/\theta )$}
\put(10,42){\vector(0,-1){27}}
\put(80,42){\vector(0,-1){27}}
\put(165,42){\vector(0,-1){27}}
\put(13,26){${\rm FC}(\theta )$}
\put(83,26){${\cal B}(u_{\theta })$}
\put(168,26){$u_{\theta }$}
\put(30,48){\vector(1,0){39}}
\put(119,48){\vector(1,0){34}}
\put(40,8){\vector(1,0){28}}
\put(129,8){\vector(1,0){24}}
\end{picture}\end{center}\vspace*{-13pt}

where the horizontal arrows represent bounded lattice embeddings (thus the ones to the left are Boolean embeddings).\end{enumerate}\label{proposition3.8}\end{proposition}

\begin{proof} (\ref{proposition3.8(1)}) By Proposition \ref{proposition3.7}, ${\rm FC}(\theta )$ is well defined. Since $u_{\theta }:{\rm Con}(A)\rightarrow {\rm Con}(A/\theta )$ is a bounded lattice morphism and ${\rm FC}(\theta )=u_{\theta }\mid _{{\rm FC}(A)}:{\rm FC}(A)\rightarrow {\rm FC}(A/\theta )$, with ${\rm FC}(A)$ a Boolean sublattice of ${\rm Con}(A)$ and ${\rm FC}(A/\theta )$ a Boolean sublattice of ${\rm Con}(A/\theta )$, it follows that ${\rm FC}(\theta )$ is a bounded lattice morphism between two Boolean algebras, hence it is a Boolean morphism.

\noindent  (\ref{proposition3.8(2)}) By the fact that ${\cal B}(u_{\theta })=u_{\theta }\mid _{{\cal B}({\rm Con}(A))}$ and ${\rm FC}(\theta )=u_{\theta }\mid _{{\rm FC}(A)}={\cal B}(u_{\theta })\mid _{{\rm FC}(A)}$.\end{proof}

\begin{definition} We say that $\theta $ has the {\em Factor Congruence Lifting Property} (abbreviated {\em FCLP}) iff the Boolean morphism ${\rm FC}(\theta ):{\rm FC}(A)\rightarrow {\rm FC}(A/\theta )$ is surjective.

We say that $A$ has the {\em Factor Congruence Lifting Property} ({\em FCLP}) iff each congruence of $A$ has the FCLP.\label{definition4.1}\end{definition}

\begin{definition}{\rm \cite{cblp}} We say that $\theta $ has the {\em Congruence Boolean Lifting Property} (abbreviated {\em CBLP}) iff the Boolean morphism ${\cal B}(u_{\theta }):{\cal B}({\rm Con}(A))\rightarrow {\cal B}({\rm Con}(A/\theta ))$ is surjective.

We say that $A$ has the {\em Congruence Boolean Lifting Property} ({\em CBLP}) iff each congruence of $A$ has the CBLP.\label{defcblp}\end{definition}

\begin{remark} The properties on CBLP that we cite in the rest of this article do hold without enforcing the hypothesis (H) from \cite{cblp}, namely the requirement that $\nabla _{A}$ is a finitely generated congruence of $A$, or, equivalently, a compact element of ${\rm Con}(A)$.\end{remark}

\begin{remark}{\rm \cite{cblp}} $\theta $ has CBLP iff the Boolean morphism ${\cal B}(v_{\theta }):{\cal B}({\rm Con}(A))\rightarrow {\cal B}([\theta ))$ is surjective. This is immediate, since $s_{\theta }$ is a bounded lattice isomorphism and thus ${\cal B}(s_{\theta })$ is a Boolean isomorphism, and we have the following commutative diagram in the category of Boolean algebras:\vspace*{-10pt}

\begin{center}\begin{picture}(120,53)(0,0)
\put(-13,30){${\cal B}({\rm Con}(A))$}
\put(35,33){\vector (1,0){40}}
\put(15,27){\vector (3,-2){33}}
\put(110,26){\vector (-3,-2){33}}
\put(42,38){${\cal B}(u_{\theta })$}
\put(5,11){${\cal B}(v_{\theta })$}
\put(96,10){${\cal B}(s_{\theta })$}
\put(76,30){${\cal B}({\rm Con}(A/\theta ))$}
\put(50,0){${\cal B}([\theta ))$}\end{picture}\end{center}\vspace*{-8pt}\label{cblpcuv}\end{remark}

\begin{lemma} $\theta $ has FCLP iff, for any $\psi \in [\theta )$ such that $\psi /\theta \in {\rm FC}(A/\theta )$, there exists a $\phi \in {\rm FC}(A)$ such that $\phi \vee \theta =\psi $.\vspace*{-20pt}

\begin{center}\begin{picture}(120,53)(0,0)
\put(0,30){${\rm FC}(A)$}
\put(29,33){\vector (1,0){45}}
\put(15,27){\vector (3,-2){33}}
\put(94,26){\vector (-3,-2){33}}
\put(39,36){${\rm FC}(\theta )$}
\put(-4,11){$v_{\theta }\mid _{{\rm FC}(A)}$}
\put(78,10){$s_{\theta }\mid _{{\rm FC}(A/\theta )}$}
\put(76,30){${\rm FC}(A/\theta )$}
\put(50,0){$[\theta )$}
\end{picture}\end{center}\vspace*{-8pt}\label{lemma4.0}\end{lemma}

\begin{proof} Let us apply the commutativity of the diagram above. By Definition \ref{definition4.1}, $\theta $ has FCLP iff, for any $\psi \in [\theta )$ such that $\psi /\theta \in {\rm FC}(A/\theta )$, there exists a $\phi \in {\rm FC}(A)$ such that $\psi /\theta ={\rm FC}(\theta )(\phi )=(\phi \vee \theta )/\theta $, that is $p_{\theta }(\psi )=p_{\theta }(\phi \vee \theta )$, that is $s_{\theta }^{-1}(\psi )=s_{\theta }^{-1}(\phi \vee \theta )$, which means that $\psi =\phi \vee \theta $ by the injectivity of the bounded lattice isomorphism $s_{\theta }^{-1}$.\end{proof}

\begin{proposition} Let $A$ be a congruence--distributive algebra. Then: $A$ has FCLP iff, for all $\phi \in {\rm Con}(A)$, $A/\phi $ has FCLP. The same goes for CBLP instead of FCLP.\label{proposition4.4}\end{proposition}

\begin{proof} The statement on CBLP is known from \cite{cblp}.

For the converse of the statement on FCLP, just take $\phi =\Delta _{A}$, so that $A/\phi =A/\Delta _{A}\cong A$.

Now assume that $A$ has FCLP, and let $\phi \in {\rm Con}(A)$ and $\psi \in [\phi )$. Let $f:{\rm Con}(A/\phi )\rightarrow {\rm Con}(A/\psi )$, for all $\alpha \in [\phi )$, $f(\alpha /\phi )=(\alpha \vee \psi )/\psi $. It is immediate that $f$ is well defined and it is a bounded lattice morphism. Let $g:{\rm Con}((A/\phi )/_{\textstyle(\psi /\phi )})\rightarrow {\rm Con}(A/\psi )$, for all $\alpha \in [\psi )\subseteq [\phi )$, $g((\alpha /\phi )/_{\textstyle(\psi /\phi )})=\alpha /\psi $. According to the Second Isomorphism Theorem (\cite{bur}), $g$ is well defined and it is a bounded lattice isomorphism. Then the following diagram in the category of bounded distributive lattices is commutative:\vspace*{-10pt}

\begin{center}\begin{picture}(200,53)(0,0)
\put(0,30){${\rm Con}(A)$}
\put(35,33){\vector(1,0){40}}
\put(15,27){\vector(3,-1){60}}
\put(94,27){\vector(0,-1){20}}
\put(96,16){$f$}
\put(50,40){$u_{\phi }$}
\put(34,10){$u_{\psi }$}
\put(76,30){${\rm Con}(A/\phi )$}
\put(162,30){${\rm Con}((A/\phi )/_{\textstyle(\psi /\phi )})$}
\put(120,33){\vector(1,0){40}}
\put(190,27){\vector(-3,-1){68}}
\put(156,9){$g$}
\put(124,40){$u_{(\psi /\phi )}$}
\put(76,0){${\rm Con}(A/\psi )$}
\end{picture}\end{center}

Indeed, for all $\alpha \in {\rm Con}(A)$, $f(u_{\phi }(\alpha ))=f((\alpha \vee \phi )/\phi )=(\alpha \vee \phi \vee \psi )/\psi =u_{\psi }(\alpha )$, and, for all $\beta \in [\phi )$, $g(u_{(\psi /\phi )}(\beta /\phi ))=g((\beta /\phi )\vee (\psi /\phi )/_{\textstyle(\psi /\phi )})=g((\beta \vee \psi )/\phi )/_{\textstyle(\psi /\phi )})=(\beta \vee \psi )/\psi =f(\beta /\phi )$. By considerring the restrictions of the morphisms in the previous diagram to the Boolean algebras of factor congruences, we obtain the following commutative diagram in the category of Boolean algebras:\vspace*{-9pt}

\begin{center}\begin{picture}(200,53)(0,0)
\put(2,30){${\rm FC}(A)$}
\put(33,33){\vector(1,0){42}}
\put(15,27){\vector(3,-1){60}}
\put(96,27){\vector(0,-1){20}}
\put(98,16){$f^{\prime }$}
\put(40,38){${\rm FC}(\phi )$}
\put(24,8){${\rm FC}(\psi )$}
\put(78,30){${\rm FC}(A/\phi )$}
\put(162,30){${\rm FC}((A/\phi )/_{\textstyle(\psi /\phi )})$}
\put(118,33){\vector(1,0){42}}
\put(190,27){\vector(-3,-1){70}}
\put(156,9){$g^{\prime }$}
\put(119,38){${\rm FC}(\psi /\phi )$}
\put(78,0){${\rm FC}(A/\psi )$}
\end{picture}\end{center}

\noindent where $f^{\prime }=f\mid _{{\rm FC}(A/\phi )}$ and $g^{\prime }=g\mid _{{\rm FC}((A/\phi )/_{\scriptstyle (\psi /\phi )})}$ both have the image within ${\rm FC}(A/\psi )$ by the very commutativity of the first diagram and the fact that the image of ${\rm FC}(\psi )$ is included in ${\rm FC}(A/\psi )$. Therefore $f^{\prime }\circ {\rm FC}(\phi )={\rm FC}(\psi )$ and $f^{\prime }=g^{\prime }\circ {\rm FC}(\psi /\phi )$. Since $A$ has FCLP, $\psi $ has FCLP, that is ${\rm FC}(\psi )$ is surjective, hence $f^{\prime }$ is surjective, thus $g^{\prime }$ is surjective. But $g$ is injective, so $g^{\prime }$ is injective. Therefore $g^{\prime }$ is bijective, so there exists $(g^{\prime })^{-1}:{\rm FC}(A/\psi )\rightarrow {\rm FC}((A/\phi )/_{\textstyle(\psi /\phi )})$, hence ${\rm FC}(\psi /\phi )=(g^{\prime })^{-1}\circ f^{\prime }$, with $(g^{\prime })^{-1}$ bijective and $f^{\prime }$ surjective, thus ${\rm FC}(\psi /\phi )$ is surjective, which means that $\psi /\phi $ has FCLP. Therefore $A/\phi $ has FCLP.\end{proof}

\begin{proposition} Let $n\in \N ^*$ and $A_1,\ldots ,A_n$ be congruence--distributive algebras and $\displaystyle A=\prod _{i=1}^nA_i$. Then: $A$ has FCLP iff each of the algebras $A_1,\ldots ,A_n$ has FCLP. The same goes for CBLP instead of FCLP.\label{proposition4.5}\end{proposition}

\begin{proof} The statement on CBLP is known from \cite{cblp}.

The direct implication in the statement on FCLP follows from Propositions \ref{proposition3.6} and \ref{proposition4.4}.

Now assume that $A_1,\ldots ,A_n$ have FCLP, and let $\phi \in {\rm Con}(A)$. By Remark \ref{4***}, there exist $\phi _1\in {\rm Con}(A_1),\ldots ,$ $\phi _n\in {\rm Con}(A_n)$ such that $\phi =\phi _1\times \ldots \times \phi _n$, and it follows that $\displaystyle A/\phi \cong \prod _{i=1}^nA_i/\phi _i$. Let $\displaystyle j:\prod _{i=1}^nA_i/\phi _i\rightarrow A/\phi $ be the isomorphism from Remark \ref{oremarca}: for all $a_1\in A_1,\ldots ,a_n\in A_n$, $j(a_1/\phi _1,\ldots ,a_n/\phi _n)=(a_1,\ldots ,a_n)/\phi $. Let $\displaystyle g:{\rm FC}(\prod _{i=1}^nA_i/\phi _i)\rightarrow {\rm FC}(A/\phi )$, defined by: for all $\displaystyle \beta \in {\rm FC}(\prod _{i=1}^nA_i/\phi _i)$, $g(\beta )=j(\beta )=\{(j(a),j(b))\ |\ (a,b)\in \beta \}$. By Remark \ref{transpcong}, $g$ is a Boolean isomorphism. Let $\displaystyle f:\prod _{i=1}^n{\rm FC}(A_i)\rightarrow {\rm FC}(A)$ and $\displaystyle h:\prod _{i=1}^n{\rm FC}(A_i/\phi _i)\rightarrow {\rm FC}(\prod _{i=1}^nA_i/\phi _i)$ be the Boolean isomorphisms from Lemma \ref{lemma3.3}: for all $\alpha _1\in {\rm FC}(A_1),\ldots ,\alpha _n\in {\rm FC}(A_n)$, $f(\alpha _1,\ldots ,\alpha _n)=\alpha _1\times \ldots \times \alpha _n$, and, for all $\gamma _1\in {\rm FC}(A_1/\phi _1),\ldots ,\gamma _n\in {\rm FC}(A_n/\phi _n)$, $h(\gamma _1,\ldots ,\gamma _n)=\gamma _1\times \ldots \times \gamma _n$. Let us denote by $\displaystyle p=\prod _{i=1}^n{\rm FC}(\phi _i):\prod _{i=1}^n{\rm FC}(A_i)\rightarrow \prod _{i=1}^n{\rm FC}(A_i/\phi _i)$, defined in the usual way: for all $\alpha _1\in {\rm FC}(A_1),\ldots ,\alpha _n\in {\rm FC}(A_n)$, $p(\alpha _1,\ldots ,\alpha _n)=({\rm FC}(\phi _1)(\alpha _1),\ldots ,{\rm FC}(\phi _n)(\alpha _n))=(u_{\phi _1})(\alpha _1),\ldots ,u_{\phi _n})(\alpha _n))=((\alpha _1\vee \phi _1)/\phi _1,\ldots ,(\alpha _n\vee \phi _n)/\phi _n)$. Then, clearly, $p$ is a Boolean morphism. The following diagram in the category of Boolean algebras is commutative:\vspace*{-10pt}

\begin{center}\begin{picture}(200,53)(0,0)
\put(2,30){${\rm FC}(A)$}
\put(33,33){\vector(1,0){42}}
\put(16,7){\vector(0,1){20}}
\put(18,13){$f$}
\put(198,13){$h$}
\put(196,7){\vector(0,1){20}}
\put(69,-18){$\displaystyle p=\prod _{i=1}^n{\rm FC}(\phi _i)$}
\put(42,0){\vector(1,0){120}}
\put(40,38){${\rm FC}(\phi )$}
\put(78,30){${\rm FC}(A/\phi )$}
\put(-5,-3){$\displaystyle \prod _{i=1}^n{\rm FC}(A_i)$}
\put(162,30){$\displaystyle {\rm FC}(\prod _{i=1}^nA_i/\phi _i)$}
\put(162,-3){$\displaystyle \prod _{i=1}^n{\rm FC}(A_i/\phi _i)$}
\put(160,33){\vector(-1,0){42}}
\put(139,38){$g$}
\end{picture}\end{center}\vspace*{26pt}

Indeed, the following hold, for all $\alpha _1\in {\rm FC}(A_1),\ldots ,\alpha _n\in {\rm FC}(A_n)$:$${\rm FC}(\phi )(f(\alpha _1,\ldots ,\alpha _n))=u_{\phi }(f(\alpha _1,\ldots ,\alpha _n))=(f(\alpha _1,\ldots ,\alpha _n)\vee \phi )/\phi =$$$$((\alpha _1\times \ldots \times \alpha _n)\vee (\phi _1\times \ldots \times \phi _n))/(\phi _1\times \ldots \times \phi _n)=$$$$((\alpha _1\vee \phi _1)\times \ldots \times ((\alpha _n\vee \phi _n))/(\phi _1\times \ldots \times \phi _n)=((\alpha _1\vee \phi _1)\times \ldots \times ((\alpha _n\vee \phi _n))/\phi $$$$\mbox{and }g(h(p(\alpha _1,\ldots ,\alpha _n)))=g(h((\alpha _1\vee \phi _1)/\phi _1,\ldots ,(\alpha _n\vee \phi _n)/\phi _n))=$$$$g((\alpha _1\vee \phi _1)/\phi _1\times \ldots \times (\alpha _n\vee \phi _n)/\phi _n)=j((\alpha _1\vee \phi _1)/\phi _1\times \ldots \times (\alpha _n\vee \phi _n)/\phi _n)=$$$$\{(j(a_1,\ldots ,a_n),j(b_1,\ldots ,b_n))\ |\ ((a_1,\ldots ,a_n),(b_1,\ldots ,b_n))\in (\alpha _1\vee \phi _1)/\phi _1\times \ldots \times (\alpha _n\vee \phi _n)/\phi _n\}=$$$$\{(j(a_1,\ldots ,a_n),j(b_1,\ldots ,b_n))\ |\ (a_1,b_1)\in (\alpha _1\vee \phi _1)/\phi _1,\ldots ,(a_n,b_n)\in (\alpha _n\vee \phi _n)/\phi _n\}=$$$$\{(j(c_1/\phi _1,\ldots ,c_n/\phi _n),j(d_1/\phi _1,\ldots ,d_n/\phi _n))\ |\ (c_1/\phi _1,d_1/\phi _1)\in (\alpha _1\vee \phi _1)/\phi _1,\ldots ,(c_n/\phi _n,d_n/\phi _n)\in (\alpha _n\vee \phi _n)/\phi _n\}$$$$=\{((c_1,\ldots ,c_n)/\phi ,(d_1,\ldots ,d_n)/\phi )\ |\ (c_1,d_1)\in \alpha _1\vee \phi _1,\ldots ,(c_n,d_n)\in \alpha _n\vee \phi _n\}=$$$$\{((c_1,\ldots ,c_n)/\phi ,(d_1,\ldots ,d_n)/\phi )\ |\ ((c_1,\ldots ,c_n),(d_1,\ldots ,d_n))\in (\alpha _1\vee \phi _1)\times \ldots \times (\alpha _n\vee \phi _n)\}=$$$$((\alpha _1\vee \phi _1)\times \ldots \times (\alpha _n\vee \phi _n))/\phi ={\rm FC}(\phi)(f(\alpha _1,\ldots ,\alpha _n)).$$Since $A_1,\ldots ,A_n$ have FCLP, it follows that $\phi _1,\ldots ,\phi _n$ have FCLP, that is ${\rm FC}(\phi _1),\ldots ,{\rm FC}(\phi _n)$ are surjective, hence $\displaystyle p=\prod _{i=1}^n{\rm FC}(\phi _i)$ is surjective. But, as we have seen, ${\rm FC}(\phi )\circ f=g\circ h\circ p$, and $f,g,h$ are bijections. Therefore ${\rm FC}(\phi )$ is surjective, which means that $\phi $ has FCLP. Thus $A$ has FCLP.\end{proof}

\begin{proposition} If $A$ is an arithmetical algebra, then:\begin{enumerate}
\item\label{proposition4.2(1)} $\theta $ has FCLP iff $\theta $ has CBLP;
\item\label{proposition4.2(2)} $A$ has FCLP iff $A$ has CBLP.\end{enumerate}\label{proposition4.2}\end{proposition}

\begin{proof} (\ref{proposition4.2(1)}) If $A$ is arithmetical, then so is $A/\theta $, thus ${\cal B}({\rm Con}(A))={\rm FC}(A)$ and ${\cal B}({\rm Con}(A/\theta ))={\rm FC}(A/\theta )$, hence ${\rm FC}(\theta )={\cal B}(u_{\theta })\mid _{{\rm FC}(A)}={\cal B}(u_{\theta })$, thus ${\rm FC}(\theta )$ is surjective iff ${\cal B}(u_{\theta })$ is surjective, that is $\theta $ has FCLP iff $\theta $ has CBLP.

\noindent (\ref{proposition4.2(2)}) By (\ref{proposition4.2(1)}).\end{proof}

\begin{lemma}{\rm \cite{cblp}}\begin{itemize}\item Any bounded distributive lattice has CBLP.
\item Any algebra from a discriminator equational class has CBLP.\end{itemize}\label{cblpabvd}\end{lemma}

\begin{proposition}\begin{itemize}
\item Any Boolean algebra has FCLP.
\item Any algebra from a discriminator equational class has FCLP.\end{itemize}\label{fclpabvd}\end{proposition}

\begin{proof} By Proposition \ref{proposition4.2}, (\ref{proposition4.2(2)}), Lemma \ref{cblpabvd} and the fact that Boolean algebras are arithmetical algebras, and, according to \cite{bj}, algebras from discriminator equational classes are arithmetical algebras, as well.\end{proof}

\begin{proposition} If $A$ is an arithmetical algebra, then:\begin{itemize}
\item $A$ is semilocal and it has CBLP iff $A$ is semilocal and ${\rm Rad}(A)$ has CBLP iff $A$ is isomorphic to a finite direct product of local algebras iff $A$ is semilocal and it has FCLP iff $A$ is semilocal and ${\rm Rad}(A)$ has FCLP;
\item $A$ is maximal and it has CBLP iff $A$ is maximal and ${\rm Rad}(A)$ has CBLP iff $A$ is isomorphic to a finite direct product of local maximal algebras iff $A$ is maximal and it has FCLP iff $A$ is maximal and ${\rm Rad}(A)$ has FCLP.\end{itemize}\end{proposition}

\begin{proof} In each of the two statements, the first two equivalences have been proven in \cite{ggcm}, and the rest follow from Proposition \ref{proposition4.2}.\end{proof}

\begin{definition}{\rm \cite{cblp}} We say that the algebra $A$ is {\em B--normal} iff, for all $\phi ,\psi \in {\rm Con}(A)$ such that $\phi \vee \psi =\nabla _{A}$, there exist $\alpha ,\beta \in {\cal B}({\rm Con}(A))$ such that $\alpha \cap \beta =\Delta _{A}$ and $\phi \vee \alpha =\psi \vee \beta =\nabla _{A}$.\end{definition}

Note that $A$ is a B--normal algebra iff ${\rm Con}(A)$ is a B--normal lattice.

\begin{proposition}{\rm \cite{cblp}} $A$ has CBLP iff $A$ is B--normal.\label{cblpbn}\end{proposition}

\begin{definition} We say that the algebra $A$ is {\em FC--normal} iff, for all $\phi ,\psi \in {\rm Con}(A)$ such that $\phi \circ \psi =\nabla _{A}$, there exist $\alpha ,\beta \in {\rm FC}(A)$ such that $\alpha \cap \beta =\Delta _{A}$ and $\phi \circ \alpha =\psi \circ \beta =\nabla _{A}$.\label{fcn}\end{definition}

By an observation in Section \ref{preliminaries}, the condition that $\phi \circ \psi =\nabla _{A}$ in Definition \ref{fcn} implies $\phi \circ \psi =\psi \circ \phi $, as well as $\phi \vee \psi =\nabla _{A}$, because $\phi \circ \psi =\nabla _{A}\in {\rm Con}(A)$ implies $\phi \vee \psi =\phi \circ \psi =\nabla _{A}$. By Remark \ref{fc}, (\ref{fc1}), the equalities $\phi \circ \alpha =\psi \circ \beta =\nabla _{A}$ in Definition \ref{fcn} are equivalent to $\phi \vee \alpha =\psi \vee \beta =\nabla _{A}$.

\begin{remark} If $A$ is an arithmetical algebra, then: $A$ is B--normal iff $A$ is FC--normal, by Remark \ref{fc}, (\ref{fc2}), and the fact that any arithmetical algebra $A$ has ${\cal B}({\rm Con}(A))={\rm FC}(A)$.\label{aritmnorm}\end{remark}

\begin{remark} $A$ is FC--normal iff, for any $\phi ,\psi \in {\rm Con}(A)$ such that $\phi \circ \psi =\nabla _{A}$, there exists an $\alpha \in {\rm FC}(A)$ such that $\phi \vee \alpha =\psi \vee \neg \, \alpha =\nabla _{A}$, which, in turn, is equivalent to $\phi \circ \alpha =\psi \circ \neg \, \alpha =\nabla _{A}$, according to Remark \ref{fc}, (\ref{fc1}). Indeed, since $\alpha \cap \neg \, \alpha =\Delta _{A}$ and $\neg \, \alpha \in {\rm FC}(A)$ for any $\alpha \in {\rm FC}(A)$, the converse implication is trivial. As for the direct implication, for any $\alpha ,\beta \in {\cal B}({\rm Con}(A))\supseteq {\rm FC}(A)$, the fact that $\alpha \cap \beta =\Delta _{A}$ means that $\beta \subseteq \neg \, \alpha $, thus, for any $\psi \in {\rm Con}(A)$, the equality $\psi \vee \beta =\nabla _{A}$ implies $\psi \vee \neg \, \alpha \supseteq \psi \vee \beta =\nabla _{A}$, thus $\psi \vee \neg \, \alpha =\nabla _{A}$.\label{fccuneg}\end{remark}

\begin{proposition} $A$ has FCLP iff A is FC--normal.\label{fclpfcn}\end{proposition}

\begin{proof} For the direct implication, assume that $A$ has FCLP, and let $\phi ,\psi \in {\rm Con}(A)$ such that $\phi \circ \psi =\nabla _{A}$. Then $\phi \cap \psi $ has FCLP, that is the Boolean morphism ${\rm FC}(\phi \cap \psi )$ is surjective. We have: $\phi ,\psi \in [\phi \cap \psi )$, $\phi /(\phi \cap \psi )\cap \psi /(\phi \cap \psi )=(\phi \cap \psi )/(\phi \cap \psi )=\Delta _{A/(\phi \cap \psi )}$ and $\phi /(\phi \cap \psi )\circ \psi /(\phi \cap \psi )=(\phi \circ \psi )/(\phi \cap \psi )=\nabla _{A}/(\phi \cap \psi )=\nabla _{A/(\phi \cap \psi )}$, hence $\phi /(\phi \cap \psi ),\psi /(\phi \cap \psi )$ is a pair of factor congruences by Remark \ref{caractfc}, (\ref{caractfc1}), so $\phi /(\phi \cap \psi ),\psi /(\phi \cap \psi )\in {\rm FC}(A/(\phi \cap \psi ))$ and $\psi /(\phi \cap \psi )=\neg \, (\phi /(\phi \cap \psi ))$. Since ${\rm FC}(\phi \cap \psi ):{\rm FC}(A)\rightarrow {\rm FC}(A/(\phi \cap \psi ))$ is surjective, it follows that there exists an $\alpha \in {\rm FC}(A)$ such that $\phi /(\phi \cap \psi )={\rm FC}(\phi \cap \psi )(\alpha )=(\alpha \vee (\phi \cap \psi ))/(\phi \cap \psi )$, and so $(\neg \, \alpha \vee (\phi \cap \psi ))/(\phi \cap \psi )={\rm FC}(\phi \cap \psi )(\neg \, \alpha )=\neg \, {\rm FC}(\phi \cap \psi )(\alpha )=\neg \, (\phi /(\phi \cap \psi ))=\psi /(\phi \cap \psi )$. Therefore $\phi =\alpha \vee (\phi \cap \psi )$ and $\psi =\neg \, \alpha \vee (\phi \cap \psi )$, thus $\phi \vee \neg \, \alpha =\psi \vee \alpha =\alpha \vee \neg \, \alpha \vee (\phi \cap \psi )=\nabla _{A}\vee (\phi \cap \psi )=\nabla _{A}$. By Remark \ref{fccuneg}, it follows that $A$ is FC--normal.

For the converse implication, assume that $A$ is FC--normal, and let $\rho \in {\rm Con}(A)$ and $\phi \in [\rho )$ such that $\phi /\rho \in {\rm FC}(A/\rho )$. Then, according to Remark \ref{caractfc}, (\ref{caractfc1}), there exists a $\psi \in [\rho )$ such that $\phi /\rho \cap \psi /\rho =\Delta _{A/\rho }$ and $\phi /\rho \circ \psi /\rho =\nabla _{A/\rho }$, so that $\phi /\rho \vee \psi /\rho =\nabla _{A/\rho }$, $\rho /\rho =\Delta _{A/\rho }=(\phi \cap \psi )/\rho $ and $\nabla _{A}/\rho =\nabla _{A/\rho }=(\phi \circ \psi )/\rho $, thus $\phi \cap \psi =\rho $, $\phi \circ \psi =\nabla _{A}$ and $\phi /\rho =\neg \, (\psi /\rho )$. By Remark \ref{fccuneg}, since $A$ is FC--normal, it follows that there exists an $\alpha \in {\rm FC}(A)$ which fulfills $\phi \vee \alpha =\psi \vee \neg \, \alpha =\nabla _{A}$, hence $\psi \supseteq \neg \, \neg \, \alpha =\alpha $. Thus $\neg \, \alpha \in {\rm FC}(A)$ and $\alpha \vee \rho =\alpha \vee (\phi \cap \psi )=(\alpha \vee \phi )\cap (\alpha \vee \psi )=\nabla _{A}\cap \psi =\psi $, so ${\rm FC}(\rho )(\alpha )=(\alpha \vee \rho )/\rho =\psi /\rho $, hence $\phi /\rho =\neg \, (\psi /\rho )=\neg \, {\rm FC}(\rho )(\alpha )={\rm FC}(\rho )(\neg \, \alpha )$. Therefore ${\rm FC}(\rho )$ is surjective, which means that $\rho $ has FCLP. Thus $A$ has FCLP.\end{proof}

\begin{corollary} If $A$ is an arithmetical algebra, then: $A$ has FCLP iff $A$ has CBLP iff $A$ is FC--normal iff $A$ is B--normal.\label{lpnorm}\end{corollary}

\begin{proof} Propositions \ref{proposition4.2}, \ref{cblpbn} and \ref{fclpfcn} and Remark \ref{aritmnorm} provide several proofs for this corollary.\end{proof}

\section{FCLP versus CBLP and BLP in Residuated Lattices and Bounded Distributive Lattices}
\label{vsblp}

Now let us study the relations between FCLP, CBLP and the Boolean Lifting
Property in residuated lattices and bounded distributive lattices. For a further study of the properties of residuated lattices that we use in what follows, we refer the reader to \cite{gal}, \cite{haj}, \cite{ior}, \cite{jits}, \cite{joh}, \cite{kow}, \cite{pic}, \cite{tur}.

We recall that a {\em (commutative) residuated lattice} is an algebra $(R,\vee ,\wedge ,\odot ,\rightarrow ,0,1)$ of type $(2,2,2,2,0,0)$, where $(R,\vee ,\wedge ,0,1)$ is a bounded lattice, $(R,\odot ,1)$ is a commutative monoid and the following property, called the {\em law of residuation}, holds for every $a,b,c\in R$: $a\odot b\leq c$ iff $a\leq b\rightarrow c$, where $\leq $ is the partial order of the lattice $(R,\vee ,\wedge )$. For any $a,b\in R$, we denote by $a\leftrightarrow b=(a\rightarrow b)\wedge (b\rightarrow a)$. For any $a\in R$ and any $n\in \N $, we denote by $a^0=1$ and $a^{n+1}=a^n\odot a$.

It is well known that residuated lattices are arithmetical algebras, and that the underlying bounded lattice of a residuated lattice $R$, although not necessarily distributive, is uniquely complemented and has the property that its set of complemented elements is a Boolean sublattice of $R$; this Boolean algebra is denoted ${\cal B}(R)$ and called the {\em Boolean center} of $R$. If $R$ and $S$ are residuated lattices and $f:R\rightarrow S$ is a residuated lattice morphism, then $f({\cal B}(R))\subseteq {\cal B}(S)$, and ${\cal B}(f)=f\mid _{{\cal B}(R)}:{\cal B}(R)\rightarrow {\cal B}(S)$ is a Boolean morphism. Thus ${\cal B}$ becomes a covariant functor from the category of residuated lattices to that of Boolean algebras. We consider that denoting this functor the same as the one from the category of bounded distributive lattices to that of Boolean algebras poses no danger of confusion.

If $R$ is a bounded distributive lattice or a residuated lattice and $\phi \in {\rm Con}(R)$, then: we say that $\phi $ fulfills the {\em Boolean Lifting Property} (abbreviated {\em BLP}) iff the Boolean morphism ${\cal B}(p_{\phi }):{\cal B}(R)\rightarrow {\cal B}(R/\phi )$ is surjective, and we say that $R$ fulfills the {\em Boolean Lifting Property} ({\em BLP}) iff all congruences of $R$ fulfill the BLP (\cite{blpiasi}, \cite{blpdacs}, \cite{ggcm}, \cite{dcggcm}). Notice that, for any $\phi \in {\rm Con}(R)$: $\phi $ has the BLP iff ${\cal B}(R/\phi )={\cal B}(R)/\phi $ iff ${\cal B}(R/\phi )\subseteq {\cal B}(R)/\phi $, since the inclusion ${\cal B}(R)/\phi \subseteq {\cal B}(R/\phi )$ always holds.

Throughout the rest of this section, $R$ shall be an arbitrary residuated lattice and $L$ shall be an arbitrary bounded distributive lattice, unless mentioned otherwise. For any $F,G\in {\rm Filt}(L)$, we denote by $F\vee G=[F\cup G)$, and, for any $I,J\in {\rm Id}(L)$, we denote by $I\vee J=(I\cup J]$. We recall that $({\rm Filt}(L),\vee ,\cap ,\{1\},L)$ and $({\rm Id}(L),\vee ,\cap ,\{0\},L)$ are bounded distributive lattices embedded in ${\rm Con}(L)$.

We recall that the {\em filters} of $R$ are the non--empty subsets of $R$ which are closed with respect to $\odot $ and to upper bounds. We shall denote by ${\rm Filt}(R)$ the set of the filters of $R$. Just as in the case of bounded distributive lattices, ${\rm Filt}(R)$ is closed with respect to arbitrary intersections, thus, for any $X\subseteq R$, there exists the {\em filter of $R$ generated by $X$}, which we shall denote by $[X)$. For any $x\in R$, $[\{x\})$ is also denoted by $[x)$ and it is called the {\em principal filter of $R$ generated by $x$}; its elements are: $[x)=\{y\in R\ |\ (\exists \, n\in \N ^*)\, (x^n\leq y)\}$. We shall denote by ${\rm PFilt}(R)$ the set of the principal filters of $R$. Just as in lattices, for any $F,G\in {\rm Filt}(R)$, we denote by $F\vee G=[F\cup G)$. $({\rm Filt}(R),\vee ,\cap ,\{1\},R)$ is a bounded distributive lattice isomorphic to ${\rm Con}(R)$, and having ${\rm PFilt}(R)$ as bounded sublattice. Let $F\in {\rm Filt}(R)$. Just as in bounded distributive lattices, $F$ is called a {\em prime filter} iff, for all $x,y\in R$, $x\vee y\in F$ implies $x\in F$ or $y\in F$, and $F$ is called a {\em maximal filter} iff it is a maximal element of ${\rm Filt}(R)\setminus \{R\}$ with respect to $\subseteq $. $R$ is called a {\em local residuated lattice} iff it has a unique maximal filter.

Let us denote by $f:{\rm Filt}(R)\rightarrow {\rm Con}(R)$ the canonical bounded lattice isomorphism, and by $g:{\rm Filt}(L)\rightarrow {\rm Con}(L)$ and $h:{\rm Id}(L)\rightarrow {\rm Con}(L)$ the canonical bounded lattice embeddings: for all $F\in {\rm Filt}(R)$, $G\in {\rm Filt}(L)$ and $I\in {\rm Id}(L)$, $f(F)=\{(x,y)\in R^2\ |\ x\leftrightarrow y\in F\}=\{(x,y)\in R^2\ |\ (\exists \, a\in F)\, (x\odot a=y\odot a)\}$, $g(G)=\{(x,y)\in L^2\ |\ (\exists \, a\in G)\, (x\wedge a=y\wedge a)\}$ and $h(I)=\{(x,y)\in L^2\ |\ (\exists \, a\in I)\, (x\vee a=y\vee a)\}$. For any $F\in {\rm Filt}(R)$, $G\in {\rm Filt}(L)$ and $I\in {\rm Id}(L)$, we say that $F$, $G$, respectively $I$, has the {\em Boolean Lifting Property (BLP)} iff the congruence $f(F)$, $g(G)$, respectively $h(I)$, has the BLP. Clearly, $R$ has the BLP iff all its filters have the BLP. We say that $L$ has the {\em Boolean Lifting Property for filters} (abbreviated {\em Filt--BLP}) iff all filters of $L$ have the BLP. We say that $L$ has the {\em Boolean Lifting Property for ideals} (abbreviated {\em Id--BLP}) iff all ideals of $L$ have the BLP (\cite{blpiasi}, \cite{blpdacs}, \cite{dcggcm}). Clearly, if $L$ has BLP, then $L$ has Filt--BLP and Id--BLP.

Let ${\cal L}$ be the {\em reticulation functor} for residuated lattices: a covariant functor from the category of residuated lattices to that of bounded distributive lattices which takes every residuated lattice $S$ to a bounded distributive lattice ${\cal L}(S)$ whose set of prime filters, endowed with the Stone topology, is homeomorphic to that of $S$ (\cite{eu3}, \cite{dcggcm}, \cite{eu1}, \cite{eu}, \cite{eu2}, \cite{eu4}, \cite{eu5}, \cite{eu7}); ${\cal L}(S)$ is uniquely determined, up to a bounded lattice isomorphism, by: ${\cal L}(S)$ is isomorphic to the dual of the bounded distributive lattice ${\rm PFilt}(S)$, thus we may take ${\cal L}(S)=({\rm PFilt}(S),\cap ,\vee ,[0)=S,[1)=\{1\})$. ${\cal L}(S)$ is called the {\em reticulation} of $S$. ${\cal L}$ has good preservation properties, which make it adequate for transferring many algebraic and topological results from the category of bounded distributive lattices to that of residuated lattices.

If we denote by $\lambda :R\rightarrow {\rm PFilt}(R)={\cal L}(R)$ the canonical surjection: for all $a\in R$, $\lambda (a)=[a)$, then the direct image of $\lambda $ sets a bijection from ${\rm Filt}(R)$ to ${\rm Filt}({\cal L}(R))$: $F\in {\rm Filt}(R)\mapsto \lambda (F)=\{\lambda (x)\ |\ x\in F\}\in {\rm Filt}({\cal L}(R))$.

\begin{lemma}{\rm \cite[Proposition $5.19$]{dcggcm}}\begin{enumerate}
\item\label{l1(1)} For any $F\in {\rm Filt}(R)$: $F$ has BLP (in $R$) iff $\lambda (F)$ has BLP (in ${\cal L}(R)$).
\item\label{l1(2)} $R$ has BLP iff ${\cal L}(R)$ has Filt--BLP.\end{enumerate}\label{l1}\end{lemma}

\begin{proposition}\begin{enumerate}
\item\label{plr1} For any $F\in {\rm Filt}(R)$: $F$ has BLP iff $f(F)$ has BLP iff $f(F)$ has CBLP iff $f(F)$ has FCLP.
\item\label{plr2} For any $\phi \in {\rm Con}(R)$: $\phi $ has BLP iff $\phi $ has CBLP iff $\phi $ has FCLP.
\item\label{plr3} $R$ has BLP iff $R$ has CBLP iff $R$ has FCLP.\end{enumerate}\label{plr}\end{proposition}

\begin{proof} (\ref{plr1}) Let $F\in {\rm Filt}(R)$. By the definition of the BLP for filters, $F$ has BLP iff $f(F)$ has BLP. In \cite{cblp}, we have proven that $F$ has BLP iff $f(F)$ has CBLP. By Proposition \ref{proposition4.2}, (\ref{proposition4.2(1)}), and the fact that $R$ is an arithmetical algebra, $f(F)$ has CBLP iff $f(F)$ has FCLP.

\noindent (\ref{plr2}) By (\ref{plr1}) and the fact that $f:{\rm Filt}(R)\rightarrow {\rm Con}(R)$ is a bijection.

\noindent (\ref{plr3}) By (\ref{plr2}).\end{proof}

\begin{proposition}\begin{enumerate}
\item\label{p0(1)} For any $\phi \in {\rm Con}(L)$: $\phi $ has BLP iff $\phi $ has FCLP.
\item\label{p0(2)} $L$ has BLP iff $L$ has FCLP.\end{enumerate}\label{p0}\end{proposition}

\begin{proof} (\ref{p0(1)}) Let $\phi \in {\rm Con}(L)$. According to Lemma \ref{lemma3.2}, the following functions are Boolean isomorphisms: $f_L:{\cal B}(L)\rightarrow {\rm FC}(L)$ and $f_{L/\phi }:{\cal B}(L/\phi )\rightarrow {\rm FC}(L/\phi )$, defined by: for all $a\in L$, if $a\in {\cal B}(L)$, then $f_L(a)=Cg_L(a,0)$, and, if $a/\phi \in {\cal B}(L/\phi )$, then $f_{L/\phi }(a/\phi )=Cg_{L/\phi }(a/\phi ,0/\phi )$. The following diagram is commutative:\vspace*{-15pt}

\begin{center}\begin{picture}(140,60)(0,0)
\put(22,35){${\cal B}(L)$}
\put(16,5){${\cal B}(L/\phi )$}
\put(111,35){${\rm FC}(L)$}
\put(104,5){${\rm FC}(L/\phi )$}
\put(7,21){${\cal B}(p_{\phi })$}
\put(33,33){\vector(0,-1){19}}
\put(123,21){${\rm FC}(\phi )$}
\put(121,33){\vector(0,-1){19}}
\put(70,42){$f_L$}
\put(45,38){\vector(1,0){65}}
\put(64,12){$f_{L/\phi }$}
\put(49,8){\vector(1,0){53}}\end{picture}\end{center}\vspace*{-4pt}

Indeed, ${\cal B}(p_{\phi })({\cal B}(L))={\cal B}(L)/\phi \subseteq {\cal B}(L/\phi )$ and, by Lemma \ref{l2.1}, for all $a\in {\cal B}(L)$, the following equalities hold: ${\rm FC}(\phi )(f_L(a))=u_{\phi }(Cg_L(a,0))=(Cg_L(a,0)\vee \phi )/\phi =Cg_{L/\phi }(a/\phi ,0/\phi )=f_{L/\phi }(a/\phi )=f_{L/\phi }(p_{\phi }(a))$. So ${\rm FC}(\phi )\circ f_L=f_{L/\phi }\circ {\cal B}(p_{\phi })$, thus, since $f_L$ and $f_{L/\phi }$ are bijections, it follows that: ${\cal B}(p_{\phi })$ is surjective iff ${\rm FC}(\phi )$ is surjective, which means that: $\phi $ has BLP iff $\phi $ has FCLP.\end{proof}

\begin{remark} Obviously, $L$ has Filt--BLP iff the dual of $L$ has Id--BLP, while, just as in the case of BLP and CBLP, $L$ has FCLP iff the dual of $L$ has FCLP, because the congruences of $L$ coincide with those of its dual.\end{remark}

\begin{corollary}\begin{enumerate}
\item\label{cl1} For any $F\in {\rm Filt}(R)$: $f(F)$ has FCLP iff $f(F)$ has CBLP iff $f(F)$ has BLP iff $\lambda (F)$ has BLP iff $g(\lambda (F))$ has BLP iff $g(\lambda (F))$ has FCLP. 
\item\label{cl2} $R$ has FCLP iff $R$ has CBLP iff $R$ has BLP iff ${\cal L}(R)$ has Filt--BLP iff each congruence in $g({\rm Filt}({\cal L}(R)))$ has BLP iff each congruence in $g({\rm Filt}({\cal L}(R)))$ has FCLP iff ${\rm PFilt}(R)$ has Id--BLP iff each congruence in $h({\rm Id}({\rm PFilt}(R)))$ has BLP iff each congruence in $h({\rm Id}({\rm PFilt}(R)))$ has FCLP.\end{enumerate}\label{cl}\end{corollary}

\begin{remark} To conclude on the above:\begin{itemize}
\item in residuated lattices, BLP, CBLP and FCLP coincide; see, in \cite{ggcm} and \cite{dcggcm}, examples of residuated lattices without BLP, thus without CBLP or FCLP, as well as examples of residuated lattices with BLP, thus with CBLP and FCLP;
\item in bounded distributive lattices, CBLP is always present, while FCLP coincides to the BLP, which implies Filt--BLP and Id--BLP; see, in \cite{blpiasi} and \cite{blpdacs}, examples of bounded distributive lattices without BLP, thus without FCLP, as well as examples of bounded distributive lattices without Filt--BLP and/or Id--BLP; thus CBLP does not imply FCLP, BLP, Filt--BLP or Id--BLP; see, also, in \cite{blpiasi} and \cite{blpdacs}, examples of bounded distributive lattices with BLP, thus with FCLP;
\item so, in residuated lattices and bounded distributive lattices, BLP and FCLP are neither always present, nor always absent.\end{itemize}\label{clcomplps}\end{remark}

\begin{corollary} CBLP does not imply FCLP.\label{compar1}\end{corollary}

\begin{proof} As pointed out in Remark \ref{clcomplps}, all bounded distributive lattices have the CBLP, but they do not all have the FCLP.\end{proof}

\begin{remark} Regarding the behaviour of the functor ${\cal L}$ with respect to these lifting properties, we conclude, by the above, that:\begin{itemize}
\item ${\cal L}$ reflects the BLP; equivalently, ${\cal L}$ reflects the FCLP;
\item ${\cal L}$ does not reflect the CBLP, as shown by the examples of residuated lattices without BLP, and thus without CBLP;
\item ${\cal L}$ preserves the CBLP, trivially;
\item ${\cal L}$ does not preserve the BLP, or, equivalently, ${\cal L}$ does not preserve the FCLP, as shown by this example of a residuated lattice with BLP, thus with FCLP, whose reticulation does not have the BLP, thus it does not have the FCLP: $R_0=\{0,a,b,c,1\}$, with the lattice structure given by the following Hasse diagram, $\odot =\wedge $ and $\rightarrow $ given by the following table: \vspace*{-10pt}

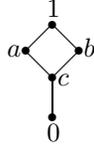
\begin{center}\begin{tabular}{cc}
\begin{picture}(100,70)(0,0)
\put(48,53){$1$}
\put(48,6){$0$}
\put(52,27){$c$}

\put(33,38){$a$}
\put(62,38){$b$}
\put(50,30){\circle*{3}}
\put(50,50){\circle*{3}}
\put(50,15){\circle*{3}}
\put(40,40){\circle*{3}}
\put(60,40){\circle*{3}}
\put(50,30){\line(1,1){10}}
\put(50,30){\line(-1,1){10}}
\put(50,50){\line(1,-1){10}}
\put(50,50){\line(-1,-1){10}}
\put(50,15){\line(0,1){15}}\end{picture}&
\begin{picture}(100,70)(0,0)
\put(0,37){\begin{tabular}{c|ccccc}
$\rightarrow $ & $0$ & $a$ & $b$ & $c$ & $1$\\ \hline 
$0$ & $1$ & $1$ & $1$ & $1$ & $1$\\ 
$a$ & $0$ & $1$ & $b$ & $b$ & $1$\\ 
$b$ & $0$ & $a$ & $1$ & $a$ & $1$\\ 
$c$ & $0$ & $1$ & $1$ & $1$ & $1$\\ 
$1$ & $0$ & $a$ & $b$ & $c$ & $1$\end{tabular}}\end{picture}\end{tabular}\end{center}\vspace*{-12pt}

$[c)=\{a,b,c,1\}$ is the unique maximal filter of $R_0$, so $R_0$ is a local residuated lattice, thus, according to \cite[Corollary $4.9$]{ggcm}, $R_0$ has BLP, thus $R_0$ has FCLP. Since $\odot =\wedge $, it follows that ${\cal L}(R_0)$ is isomorphic to the underlying bounded lattice of $R_0$ (\cite{eu}, \cite{eu2}), which, according to \cite[Example $2$]{blpiasi}, does not have Id--BLP, thus it does not have BLP, so it does not have FCLP.\end{itemize}\label{clllps}\end{remark}

\section{FCLP versus CBLP, with Examples}
\label{examples}

For the properties related to lattices that we recall in this section, see \cite{bal}, \cite{blyth}, \cite{gratzer}.

We have seen that, in arithmetical algebras, FCLP and CBLP coincide. In Section \ref{vsblp}, we have seen that, in general, they differ (Corollary \ref{compar1}). In this section we shall see that, moreover, in general, FCLP and CBLP are independent of each other. We shall do this by obtaining some properties that are useful in calculations and then providing some examples in lattices. But, first, just to show that the following results are not trivial, let us notice that the lattice structure of the set of congruences of a congruence--distributive algebra does not determine its factor congruences, also by examples in lattices.

First, let us notice that, in lattices, FCLP and CBLP are self--dual, that is a lattice $L$ has FCLP or CBLP iff its dual has FCLP or CBLP, respectively, which can be easily seen, for instance, from Remark \ref{transpcong}.

It is well known that, if $L$ is a finite distributive lattice, then its lattice of congruences is a Boolean algebra, hence ${\cal B}({\rm Con}(L))={\rm Con}(L)$. If $B$ is a Boolean algebra, then $B$ is an arithmetical algebra, thus ${\rm FC}(B)={\cal B}({\rm Con}(B))$. If $B$ is a finite Boolean algebra, then, by the above and/or the well--known fact that, in this case, ${\rm Con}(B)$ is isomorphic to $B$, it follows that ${\rm FC}(B)={\cal B}({\rm Con}(B))={\rm Con}(B)\cong B$. It is also well known that the classes of any congruence of a lattice $L$ are convex sublattices of $L$.

\begin{remark} It is clear that, if $L$ is a lattice, $S$ is a sublattice of $L$ and $\theta \in {\rm Con}(L)$, then $\theta \cap S^2\in {\rm Con}(S)$.\label{congrsl1}\end{remark}

For any $n\in \N ^*$, we shall denote by ${\cal L}_n$ the $n$--element chain. We shall denote by  ${\cal D}$ the diamond and by ${\cal P}$ the pentagon. We shall use the following ad--hoc notation: for any set $M$ and any partition $\pi $ of $M$, we denote by $eq(\pi )$ the equivalence on $M$ that corresponds to $\pi $; if $\pi $ is finite: $\pi =\{M_1,\ldots ,M_n\}$ for some $n\in \N ^*$, then $eq(\pi )$ shall be denoted, simply, by $eq(M_1,\ldots ,M_n)$. For any lattice $L$ with $1$ and any lattice $M$ with $0$, we shall denote by $L\dotplus M$ the ordinal sum between $L$ and $M$ and, for any $\phi \in {\rm Con}(L)$ and any $\psi \in {\rm Con}(M)$, by $\phi \dotplus \psi =eq((L/\phi \setminus c/\phi )\cup \{c/\phi \cup c/\psi \}\cup (M/\psi \setminus c/\psi ))$, where $c$ is the common element of $L$ and $M$ in $L\dotplus M$.

\begin{remark} Let $L$ be a lattice with $1$ and $M$ be a lattice with $0$. Then, by a result in \cite{cblp}:\begin{itemize}
\item the mapping $\phi \dotplus \psi \mapsto \phi \times \psi $ is a bounded lattice isomorphism between ${\rm Con}(L\dotplus M)$ and ${\rm Con}(L\times M)$ and thus a Boolean isomorphism between ${\cal B}({\rm Con}(L\dotplus M))$ and ${\cal B}({\rm Con}(L\times M))$, so ${\rm Con}(L\dotplus M)=\{\phi \dotplus \psi \ |\ \phi \in {\rm Con}(L),\psi \in {\rm Con}(M)\}\cong {\rm Con}(L\times M)\cong {\rm Con}(L)\times {\rm Con}(M)$ and ${\cal B}({\rm Con}(L\dotplus M))=\{\phi \dotplus \psi \ |\ \phi \in {\cal B}({\rm Con}(L)),\psi \in {\cal B}({\rm Con}(M))\}\cong {\cal B}({\rm Con}(L\times M))\cong {\cal B}({\rm Con}(L)\times {\rm Con}(M))\cong {\cal B}({\rm Con}(L))\times {\cal B}({\rm Con}(M))$;
\item $(L\dotplus M)/(\Delta _L\dotplus \nabla _M)\cong L$ and $(L\dotplus M)/(\nabla _L\dotplus \Delta _M)\cong M$.\end{itemize}

However, ${\rm FC}(L\dotplus M)$ is not necessarily isomorphic to ${\rm FC}(L\times M)\cong {\rm FC}(L)\times {\rm FC}(M)$, and ${\rm FC}(L\dotplus M)$ and $\{\phi \dotplus \psi \ |\ \phi \in {\rm FC}(L),\psi \in {\rm FC}(M)\}$ are not necessarily equal, as shown by the example of the bounded lattice $Z={\cal P}\dotplus {\cal L}_2^2$ in Remark \ref{z} below.\label{congrsl2}\end{remark}

\begin{corollary} Let $L$ be a lattice with $1$, $M$ a lattice with $0$ and $K$ a bounded lattice. If $L\dotplus M$ has FCLP, then $L$ and $M$ have FCLP. If $L\dotplus K\dotplus M$ has FCLP, then $L$, $K$ and $M$ have FCLP. The converses of these implications do not hold. The same goes for CBLP instead of FCLP.\label{nicecor}\end{corollary}

\begin{proof} The statements on CBLP are known from \cite{cblp}. The statement on FCLP for $L\dotplus M$ follows from Remark \ref{congrsl2} and Proposition \ref{proposition4.4}. The one on $L\dotplus K\dotplus M$ follows from the previous one and the associativity of $\dotplus $. Example \ref{nd} below contradicts the converse implications, because $X={\cal L}_2^2\dotplus {\cal D}$ in this example does not have FCLP, despite the fact that, according to Example \ref{ex4.14}, both ${\cal L}_2^2$ and ${\cal D}$ have FCLP.\end{proof}

\begin{remark} The converses of the implications in Corollary \ref{nicecor} do not hold in bounded distributive lattices either. Indeed, by Example \ref{ex4.14}, ${\cal L}_2$ and ${\cal L}_2^2$ have FCLP, but, by Remark \ref{clllps}, ${\cal L}_2\dotplus {\cal L}_2^2$ does not have FCLP.\end{remark}

\begin{example} Let us draw the Hasse diagrams of the chains ${\cal L}_2$ and ${\cal L}_3$, the Boolean algebras ${\cal L}_2^2$ and ${\cal L}_2^3$, the bounded distributive lattice ${\cal L}_2\times {\cal L}_3$ and the bounded non--distributive lattices ${\cal D}$ and ${\cal P}$:\vspace*{-10pt}

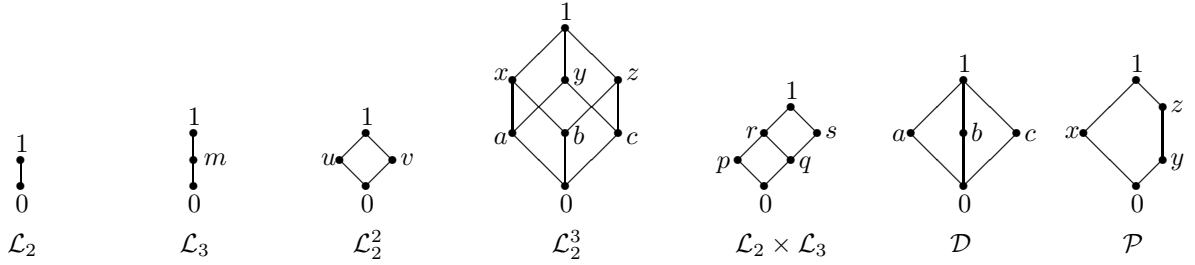
\begin{center}\begin{tabular}{ccccccc}
\begin{picture}(40,100)(0,0)
\put(20,25){\line(0,1){10}}
\put(20,25){\circle*{3}}
\put(20,35){\circle*{3}}
\put(18,15){$0$}
\put(18,38){$1$}
\put(15,0){${\cal L}_2$}
\end{picture}
&\hspace*{10pt}
\begin{picture}(40,100)(0,0)
\put(20,25){\line(0,1){20}}
\put(20,25){\circle*{3}}
\put(20,35){\circle*{3}}
\put(20,45){\circle*{3}}
\put(18,15){$0$}
\put(24,33){$m$}
\put(18,48){$1$}

\put(15,0){${\cal L}_3$}
\end{picture}
&\hspace*{10pt}
\begin{picture}(40,100)(0,0)
\put(20,25){\line(1,1){10}}
\put(20,25){\line(-1,1){10}}
\put(20,45){\line(1,-1){10}}
\put(20,45){\line(-1,-1){10}}
\put(20,25){\circle*{3}}
\put(3,33){$u$}
\put(33,33){$v$}
\put(10,35){\circle*{3}}

\put(30,35){\circle*{3}}
\put(20,45){\circle*{3}}
\put(18,15){$0$}
\put(18,48){$1$}
\put(15,0){${\cal L}_2^2$}
\end{picture}
&\hspace*{10pt}
\begin{picture}(60,100)(0,0)
\put(30,25){\line(0,1){20}}
\put(10,45){\line(0,1){20}}
\put(50,45){\line(0,1){20}}
\put(30,65){\line(0,1){20}}

\put(30,25){\line(1,1){20}}
\put(30,25){\line(-1,1){20}}
\put(30,45){\line(1,1){20}}
\put(30,45){\line(-1,1){20}}
\put(30,65){\line(1,-1){20}}
\put(30,65){\line(-1,-1){20}}

\put(30,85){\line(1,-1){20}}
\put(30,85){\line(-1,-1){20}}
\put(30,25){\circle*{3}}
\put(30,45){\circle*{3}}
\put(10,45){\circle*{3}}
\put(50,45){\circle*{3}}
\put(10,65){\circle*{3}}
\put(50,65){\circle*{3}}
\put(30,65){\circle*{3}}
\put(30,85){\circle*{3}}
\put(28,15){$0$}
\put(28,88){$1$}
\put(3,41){$a$}
\put(3,65){$x$}
\put(33,41){$b$}
\put(33,65){$y$}
\put(53,41){$c$}
\put(53,65){$z$}
\put(25,0){${\cal L}_2^3$}
\end{picture}
&\hspace*{10pt}
\begin{picture}(40,100)(0,0)
\put(20,25){\line(1,1){20}}
\put(20,25){\line(-1,1){10}}
\put(20,45){\line(1,-1){10}}
\put(30,55){\line(-1,-1){20}}
\put(30,55){\line(1,-1){10}}
\put(20,25){\circle*{3}}
\put(10,35){\circle*{3}}
\put(30,35){\circle*{3}}
\put(20,45){\circle*{3}}
\put(40,45){\circle*{3}}
\put(30,55){\circle*{3}}
\put(2,32){$p$}
\put(33,32){$q$}
\put(13,43){$r$}
\put(43,43){$s$}
\put(18,15){$0$}
\put(28,58){$1$}
\put(9,0){${\cal L}_2\times {\cal L}_3$}
\end{picture}
&\hspace*{10pt}
\begin{picture}(40,100)(0,0)
\put(30,25){\line(0,1){40}}
\put(30,25){\line(1,1){20}}
\put(30,25){\line(-1,1){20}}
\put(30,65){\line(1,-1){20}}
\put(30,65){\line(-1,-1){20}}
\put(30,25){\circle*{3}}
\put(10,45){\circle*{3}}
\put(3,42){$a$}
\put(30,45){\circle*{3}}
\put(33,42){$b$}
\put(50,45){\circle*{3}}
\put(53,42){$c$}
\put(30,65){\circle*{3}}
\put(28,15){$0$}
\put(28,68){$1$}
\put(25,0){${\cal D}$}
\end{picture}
&\hspace*{10pt}
\begin{picture}(40,100)(0,0)
\put(30,25){\line(1,1){10}}
\put(30,25){\line(-1,1){20}}
\put(30,65){\line(1,-1){10}}
\put(30,65){\line(-1,-1){20}}
\put(40,35){\line(0,1){20}}
\put(30,25){\circle*{3}}
\put(10,45){\circle*{3}}
\put(3,43){$x$}
\put(40,35){\circle*{3}}
\put(43,33){$y$}
\put(30,65){\circle*{3}}

\put(40,55){\circle*{3}}
\put(43,53){$z$}
\put(25,0){${\cal P}$}
\put(28,15){$0$}
\put(28,68){$1$}
\end{picture}
\end{tabular}
\end{center}

Let us also consider the following bounded non--distributive lattices:\vspace*{-13pt}

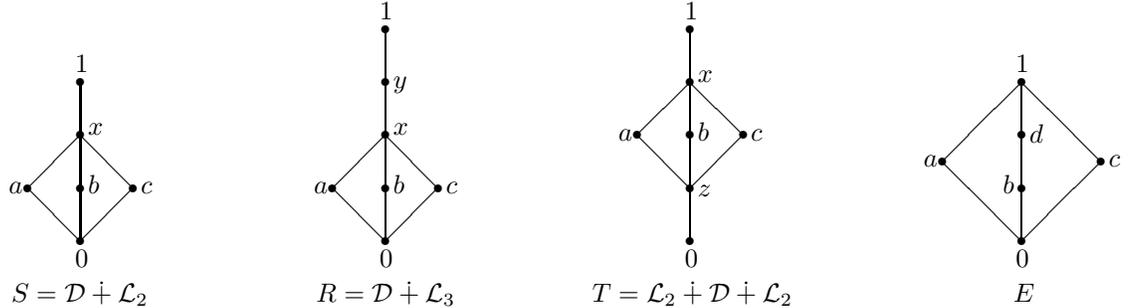
\begin{center}\begin{tabular}{cccc}
\begin{picture}(40,120)(0,0)
\put(30,20){\line(0,1){60}}
\put(30,20){\line(1,1){20}}
\put(30,20){\line(-1,1){20}}
\put(30,60){\line(1,-1){20}}
\put(30,60){\line(-1,-1){20}}
\put(30,20){\circle*{3}}
\put(10,40){\circle*{3}}

\put(30,40){\circle*{3}}
\put(50,40){\circle*{3}}
\put(30,60){\circle*{3}}
\put(30,80){\circle*{3}}
\put(3,38){$a$}
\put(33,38){$b$}
\put(33,60){$x$}
\put(53,38){$c$}
\put(28,10){$0$}
\put(28,84){$1$}
\put(4,-3){$S={\cal D}\dotplus {\cal L}_2$}
\end{picture}
&\hspace*{60pt}
\begin{picture}(40,120)(0,0)
\put(3,38){$a$}
\put(33,38){$b$}
\put(33,60){$x$}
\put(33,78){$y$}
\put(53,38){$c$}
\put(28,10){$0$}
\put(28,104){$1$}
\put(30,20){\line(0,1){80}}
\put(30,20){\line(1,1){20}}
\put(30,20){\line(-1,1){20}}
\put(30,60){\line(1,-1){20}}
\put(30,60){\line(-1,-1){20}}
\put(30,20){\circle*{3}}
\put(10,40){\circle*{3}}
\put(30,40){\circle*{3}}
\put(50,40){\circle*{3}}
\put(30,60){\circle*{3}}
\put(30,80){\circle*{3}}

\put(30,100){\circle*{3}}
\put(4,-3){$R={\cal D}\dotplus {\cal L}_3$}
\end{picture}
&\hspace*{60pt}
\begin{picture}(40,120)(0,0)
\put(30,20){\line(0,1){80}}
\put(30,40){\line(1,1){20}}
\put(30,40){\line(-1,1){20}}
\put(30,80){\line(1,-1){20}}
\put(30,80){\line(-1,-1){20}}
\put(30,40){\circle*{3}}
\put(10,60){\circle*{3}}
\put(30,60){\circle*{3}}
\put(50,60){\circle*{3}}
\put(30,80){\circle*{3}}
\put(30,100){\circle*{3}}
\put(30,20){\circle*{3}}
\put(3,58){$a$}

\put(33,58){$b$}
\put(53,58){$c$}
\put(28,10){$0$}
\put(28,104){$1$}
\put(33,80){$x$}
\put(33,37){$z$}
\put(-7,-3){$T={\cal L}_2\dotplus {\cal D}\dotplus {\cal L}_2$}
\end{picture}
&\hspace*{60pt}
\begin{picture}(40,120)(0,0)
\put(40,20){\line(0,1){60}}
\put(40,20){\line(1,1){30}}
\put(40,20){\line(-1,1){30}}
\put(40,80){\line(1,-1){30}}
\put(40,80){\line(-1,-1){30}}
\put(40,20){\circle*{3}}
\put(40,40){\circle*{3}}
\put(40,60){\circle*{3}}
\put(40,80){\circle*{3}}
\put(10,50){\circle*{3}}
\put(70,50){\circle*{3}}
\put(3,48){$a$}
\put(33,38){$b$}
\put(43,57){$d$}
\put(73,48){$c$}
\put(38,10){$0$}
\put(38,84){$1$}
\put(37,-3){$E$}\end{picture}\end{tabular}\end{center}

Now let us determine all the congruences, as well as the Boolean and the factor congruences of the lattices above. The first five of these are finite distributive lattices, so all their congruences are Boolean; out of these lattices, the ones which are not Boolean algebras, namely ${\cal L}_3$ and ${\cal L}_2\times {\cal L}_3$, have Boolean congruences which are not factor congruences; as for the two finite non--distributive lattices, ${\cal D}$ has only the two congruences which are present in any algebra, and which are factor congruences, while ${\cal P}$ has five congruences, three of which are not Boolean. See in Example \ref{ex} below finite non--distributive lattices with Boolean congruences that are not factor congruences.

For any $n\in \N $, ${\cal L}_2^n$ is a finite Boolean algebra, thus ${\rm FC}({\cal L}_2^n)={\cal B}({\rm Con}({\cal L}_2^n))={\rm Con}({\cal L}_2^n)\cong {\cal L}_2^n$. Hence ${\rm FC}({\cal L}_2)={\cal B}({\rm Con}({\cal L}_2))={\rm Con}({\cal L}_2)=\{\Delta _{{\cal L}_2},\nabla _{{\cal L}_2}\}$ and ${\rm FC}({\cal L}_2^2)={\cal B}({\rm Con}({\cal L}_2^2))={\rm Con}({\cal L}_2^2)=\{\Delta _{{\cal L}_2^2},eq(\{0,u\},\{v,1\}),$\linebreak $eq(\{0,v\},\{u,1\}),\nabla _{{\cal L}_2^2}\}$, where the last equality is straightforward. Let us denote by $\phi =eq(\{0,u\},\{v,1\})$ and by $\theta =eq(\{0,v\},\{u,1\})=\neg \, \phi $. In order to determine the congruences of ${\cal L}_2^3$, we can simply calculate the congruence of ${\cal L}_2^3$ associated to each of its eight filters/ideals, which are all principal; we obtain: ${\rm FC}({\cal L}_2^3)={\cal B}({\rm Con}({\cal L}_2^3))={\rm Con}({\cal L}_2^3)=\{\Delta _{{\cal L}_2^3},eq(\{0,c\},\{a,y\},\{b,z\},\{x,1\}),eq(\{0,b\},\{a,x\},\{c,z\},\{y,1\}),eq(\{0,a\},\{b,x\},\{c,y\},\{z,1\}),$\linebreak $eq(\{0,b,c,z\},\{a,x,y,1\}),eq(\{0,a,c,y\},\{b,x,z,1\}),eq(\{0,a,b,x\},\{c,y,z,1\}),\nabla _{{\cal L}_2^3}\}$.

It is immediate that ${\rm Con}({\cal L}_3)=\{\Delta _{{\cal L}_3},eq(\{0,m\},\{1\}),eq(\{0\},\{m,1\}),\nabla _{{\cal L}_3}\}\cong {\cal L}_2^2$, so it is a Boolean algebra, thus ${\cal B}({\rm Con}({\cal L}_3))={\rm Con}({\cal L}_3)$. If we denote by $\varphi =eq(\{0,m\},\{1\})$ and by $\psi =eq(\{0\},\{m,1\})=\neg \, \varphi $, then we notice that, for instance, $(0,1)\in \psi \circ \varphi $ and $(0,1)\notin \varphi \circ \psi $, thus $\psi \circ \varphi \neq \varphi \circ \psi $, hence $\varphi ,\psi \notin {\rm FC}({\cal L}_3)$, therefore ${\rm FC}({\cal L}_3)=\{\Delta _{{\cal L}_3},\nabla _{{\cal L}_3}\}\cong {\cal L}_2$.\vspace*{-12pt}

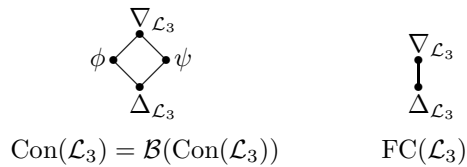
\begin{center}\begin{tabular}{cc}
\begin{picture}(40,60)(0,0)
\put(20,25){\line(1,1){10}}
\put(20,25){\line(-1,1){10}}
\put(20,45){\line(1,-1){10}}
\put(20,45){\line(-1,-1){10}}
\put(20,25){\circle*{3}}
\put(1,33){$\phi $}
\put(33,33){$\psi $}
\put(10,35){\circle*{3}}
\put(30,35){\circle*{3}}
\put(20,45){\circle*{3}}
\put(16,15){$\Delta _{{\cal L}_3}$}
\put(16,48){$\nabla _{{\cal L}_3}$}
\put(-29,-2){${\rm Con}({\cal L}_3)={\cal B}({\rm Con}({\cal L}_3))$}
\end{picture}
&\hspace*{50pt}
\begin{picture}(40,60)(0,0)

\put(20,25){\line(0,1){10}}
\put(20,25){\circle*{3}}
\put(20,35){\circle*{3}}
\put(16,15){$\Delta _{{\cal L}_3}$}
\put(16,38){$\nabla _{{\cal L}_3}$}

\put(6,-2){${\rm FC}({\cal L}_3)$}
\end{picture}\end{tabular}\end{center}

By Remark \ref{oremarca}, it follows that ${\rm Con}({\cal L}_2\times {\cal L}_3)=\{\theta _1\times \theta _2\ |\ \theta _1\in {\rm Con}({\cal L}_2),\theta _2\in {\rm Con}({\cal L}_3)\}=\{\Delta _{{\cal L}_2\times {\cal L}_3}=\Delta _{{\cal L}_2}\times \Delta _{{\cal L}_3},\Delta _{{\cal L}_2}\times \varphi ,\Delta _{{\cal L}_2}\times \psi ,\Delta _{{\cal L}_2}\times \nabla _{{\cal L}_3},\nabla _{{\cal L}_2}\times \Delta _{{\cal L}_3},\nabla _{{\cal L}_2}\times \varphi ,\nabla _{{\cal L}_2}\times \psi ,\nabla _{{\cal L}_2\times {\cal L}_3}=\nabla _{{\cal L}_2}\times \nabla _{{\cal L}_3}\}\cong {\cal L}_2^3$, because, if we denote by $\lambda _0=\Delta _{{\cal L}_2}\times \varphi $, $\lambda _1=\Delta _{{\cal L}_2}\times \psi $, $\lambda =\Delta _{{\cal L}_2}\times \nabla _{{\cal L}_3}$, $\mu =\nabla _{{\cal L}_2}\times \Delta _{{\cal L}_3}$, $\mu _0=\nabla _{{\cal L}_2}\times \varphi $, and $\mu _1=\nabla _{{\cal L}_2}\times \psi $, then we notice that the lattice structure of ${\rm Con}({\cal L}_2\times {\cal L}_3)$ is the following:\vspace*{-5pt}

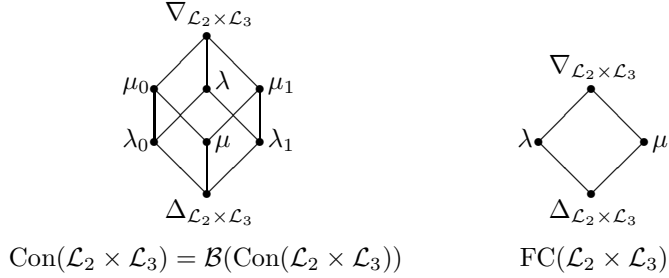
\begin{center}\begin{tabular}{cc}
\begin{picture}(60,100)(0,0)
\put(30,25){\line(0,1){20}}
\put(10,45){\line(0,1){20}}
\put(50,45){\line(0,1){20}}
\put(30,65){\line(0,1){20}}
\put(30,25){\line(1,1){20}}
\put(30,25){\line(-1,1){20}}
\put(30,45){\line(1,1){20}}
\put(30,45){\line(-1,1){20}}
\put(30,65){\line(1,-1){20}}
\put(30,65){\line(-1,-1){20}}
\put(30,85){\line(1,-1){20}}

\put(30,85){\line(-1,-1){20}}
\put(30,25){\circle*{3}}
\put(30,45){\circle*{3}}
\put(10,45){\circle*{3}}
\put(50,45){\circle*{3}}
\put(10,65){\circle*{3}}
\put(50,65){\circle*{3}}
\put(30,65){\circle*{3}}
\put(30,85){\circle*{3}}
\put(14,16){$\Delta _{{\cal L}_2\times {\cal L}_3}$}
\put(14,91){$\nabla _{{\cal L}_2\times {\cal L}_3}$}
\put(-2,43){$\lambda _0$}
\put(-2,65){$\mu _0$}
\put(33,43){$\mu $}
\put(33,65){$\lambda $}
\put(53,43){$\lambda _1$}
\put(53,65){$\mu _1$}
\put(-45,-2){${\rm Con}({\cal L}_2\times {\cal L}_3)={\cal B}({\rm Con}({\cal L}_2\times {\cal L}_3))$}\end{picture}
&\hspace*{70pt}

\begin{picture}(60,100)(0,0)
\put(30,25){\line(1,1){20}}

\put(30,25){\line(-1,1){20}}
\put(30,65){\line(1,-1){20}}
\put(30,65){\line(-1,-1){20}}
\put(30,25){\circle*{3}}
\put(10,45){\circle*{3}}
\put(50,45){\circle*{3}}
\put(30,65){\circle*{3}}
\put(14,16){$\Delta _{{\cal L}_2\times {\cal L}_3}$}
\put(14,71){$\nabla _{{\cal L}_2\times {\cal L}_3}$}
\put(2,43){$\lambda $}
\put(53,43){$\mu $}
\put(3,-2){${\rm FC}({\cal L}_2\times {\cal L}_3)$}
\end{picture}\end{tabular}\end{center}

Hence ${\cal B}({\rm Con}({\cal L}_2\times {\cal L}_3))={\rm Con}({\cal L}_2\times {\cal L}_3)\cong {\cal L}_2^3$. It is easy to see that: $\lambda _0=eq(\{0,q\},\{p,r\},\{s\},\{1\})$, $\lambda _1=eq(\{0\},\{p\},\{q,s\},\{r,1\})$, $\lambda =eq(\{0,q,s\},\{p,r,1\})$, $\mu =eq(\{0,p\},\{q,r\},\{s,1\})$, $\mu _0=eq(\{0,p,q,r\},\{s,1\})$ and $\mu _1=eq(\{0,p\},\{q,r,s,1\})$. By Lemma \ref{lemma3.3}, ${\rm FC}({\cal L}_2\times {\cal L}_3)=\{\theta _1\times \theta _2\ |\ \theta _1\in {\rm FC}({\cal L}_2),\theta _2\in {\rm FC}({\cal L}_3)\}=\{\Delta _{{\cal L}_2}\times \Delta _{{\cal L}_3},\Delta _{{\cal L}_2}\times \nabla _{{\cal L}_3},\nabla _{{\cal L}_2}\times \Delta _{{\cal L}_3},\nabla _{{\cal L}_2}\times \nabla _{{\cal L}_3}\}=\{\Delta _{{\cal L}_2\times {\cal L}_3},\lambda ,\mu ,\nabla _{{\cal L}_2\times {\cal L}_3}\}\cong {\cal L}_2^2$.

It is immediate that ${\rm Con}({\cal D})=\{\Delta _{\cal D},\nabla _{\cal D}\}$, thus ${\rm FC}({\cal D})={\cal B}({\rm Con}({\cal D}))={\rm Con}({\cal D})=\{\Delta _{\cal D},\nabla _{\cal D}\}\cong {\cal L}_2$.

It is easy to notice that ${\rm Con}({\cal P})=\{\Delta _{\cal P},\alpha ,\beta ,\gamma ,\nabla _{\cal P}\}$, where $\alpha =eq(\{0,y,z\},\{x,1\})$, $\beta =eq(\{0,x\},\{y,z,1\})$ and $\gamma =eq(\{0\},\{x\},\{y,z\},\{1\})$, and thus the lattice structure of ${\rm Con}({\cal P})$ is the following:\vspace*{-5pt}

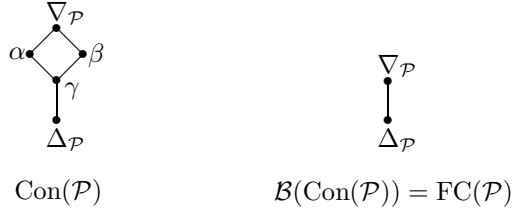
\begin{center}\begin{tabular}{cc}
\begin{picture}(40,80)(0,0)
\put(20,45){\line(1,1){10}}
\put(20,45){\line(-1,1){10}}
\put(20,65){\line(1,-1){10}}
\put(20,65){\line(-1,-1){10}}
\put(20,45){\line(0,-1){15}}
\put(16,20){$\Delta _{\cal P}$}
\put(16,68){$\nabla _{\cal P}$}
\put(20,30){\circle*{3}}
\put(23,40){$\gamma $}
\put(2,53){$\alpha $}
\put(32,52){$\beta $}
\put(20,45){\circle*{3}}
\put(10,55){\circle*{3}}
\put(30,55){\circle*{3}}
\put(20,65){\circle*{3}}
\put(4,0){${\rm Con}({\cal P})$}
\end{picture}
&\hspace*{70pt}
\begin{picture}(40,80)(0,0)
\put(20,30){\line(0,1){15}}
\put(20,30){\circle*{3}}
\put(20,45){\circle*{3}}
\put(16,20){$\Delta _{\cal P}$}
\put(16,48){$\nabla _{\cal P}$}
\put(-23,0){${\cal B}({\rm Con}({\cal P}))={\rm FC}({\cal P})$}
\end{picture}\end{tabular}\end{center}\vspace*{-3pt}

Hence ${\cal B}({\rm Con}({\cal P}))=\{\Delta _{\cal P},\nabla _{\cal P}\}$, and thus ${\rm FC}({\cal P})={\cal B}({\rm Con}({\cal P}))=\{\Delta _{\cal P},\nabla _{\cal P}\}\cong {\cal L}_2$.

The following calculations are straightforward, where we are determining the congruences by using Remarks \ref{congrsl2} and \ref{congrsl1} (the latter for obtaining ${\rm Con}(E)$), and the structure of ${\rm Con}({\cal D})$, obtained in Example \ref{ex}.\vspace*{-7pt}

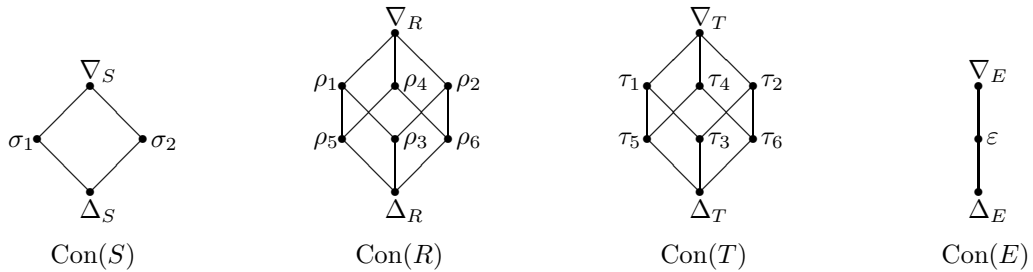
\begin{center}\begin{tabular}{cccc}
\begin{picture}(60,100)(0,0)
\put(30,25){\line(1,1){20}}
\put(30,25){\line(-1,1){20}}
\put(30,65){\line(1,-1){20}}
\put(30,65){\line(-1,-1){20}}
\put(30,25){\circle*{3}}
\put(10,45){\circle*{3}}
\put(50,45){\circle*{3}}
\put(30,65){\circle*{3}}
\put(26,15){$\Delta _S$}
\put(26,68){$\nabla _S$}
\put(-1,42){$\sigma _1$}
\put(53,42){$\sigma _2$}
\put(15,-2){${\rm Con}(S)$}\end{picture}
&\hspace*{40pt}

\begin{picture}(60,100)(0,0)
\put(30,25){\line(0,1){20}}
\put(10,45){\line(0,1){20}}
\put(50,45){\line(0,1){20}}
\put(30,65){\line(0,1){20}}
\put(30,25){\line(1,1){20}}
\put(30,25){\line(-1,1){20}}
\put(30,45){\line(1,1){20}}
\put(30,45){\line(-1,1){20}}
\put(30,65){\line(1,-1){20}}
\put(30,65){\line(-1,-1){20}}
\put(30,85){\line(1,-1){20}}
\put(30,85){\line(-1,-1){20}}
\put(30,25){\circle*{3}}
\put(30,45){\circle*{3}}

\put(10,45){\circle*{3}}
\put(50,45){\circle*{3}}
\put(10,65){\circle*{3}}
\put(50,65){\circle*{3}}
\put(30,65){\circle*{3}}
\put(30,85){\circle*{3}}
\put(26,15){$\Delta _R$}
\put(26,88){$\nabla _R$}
\put(-1,43){$\rho _5$}
\put(-1,65){$\rho _1$}

\put(33,43){$\rho _3$}
\put(33,65){$\rho _4$}
\put(53,43){$\rho _6$}
\put(53,65){$\rho _2$}
\put(15,-2){${\rm Con}(R)$}\end{picture}
&\hspace*{40pt}
\begin{picture}(60,100)(0,0)
\put(30,25){\line(0,1){20}}
\put(10,45){\line(0,1){20}}
\put(50,45){\line(0,1){20}}
\put(30,65){\line(0,1){20}}
\put(30,25){\line(1,1){20}}
\put(30,25){\line(-1,1){20}}
\put(30,45){\line(1,1){20}}
\put(30,45){\line(-1,1){20}}
\put(30,65){\line(1,-1){20}}
\put(30,65){\line(-1,-1){20}}
\put(30,85){\line(1,-1){20}}
\put(30,85){\line(-1,-1){20}}
\put(30,25){\circle*{3}}
\put(30,45){\circle*{3}}
\put(10,45){\circle*{3}}
\put(50,45){\circle*{3}}
\put(10,65){\circle*{3}}
\put(50,65){\circle*{3}}
\put(30,65){\circle*{3}}
\put(30,85){\circle*{3}}
\put(26,15){$\Delta _T$}
\put(26,88){$\nabla _T$}
\put(-1,43){$\tau _5$}
\put(-1,65){$\tau _1$}
\put(33,43){$\tau _3$}
\put(33,65){$\tau _4$}
\put(53,43){$\tau _6$}
\put(53,65){$\tau _2$}
\put(15,-2){${\rm Con}(T)$}\end{picture}
&\hspace*{40pt}
\begin{picture}(40,100)(0,0)

\put(20,25){\line(0,1){40}}
\put(20,25){\circle*{3}}
\put(20,45){\circle*{3}}
\put(20,65){\circle*{3}}

\put(16,15){$\Delta _E$}
\put(23,43){$\varepsilon $}
\put(16,68){$\nabla _E$}
\put(6,-2){${\rm Con}(E)$}
\end{picture}
\end{tabular}\end{center}

${\rm Con}(S)=\{\Delta _S,\sigma _1,\sigma _2,\nabla _S\}$, where $\sigma _1=eq(\{0\},\{a\},\{b\},\{c\},\{x,1\})$ and $\sigma _2=eq(\{0,a,b,c,x\},\{1\})$, hence ${\cal B}({\rm Con}(S))={\rm Con}(S)\cong {\cal L}_2^2$, in which $\sigma _2=\neg \, \sigma _1$. Since $(1,b)\in \sigma _2\circ \sigma _1$, but $(1,b)\notin \sigma _1\circ \sigma _2$, it follows that $\sigma _1\circ \sigma _2\neq \sigma _2\circ \sigma _1$, thus $\sigma _1,\sigma _2\notin {\rm FC}(S)$, so ${\rm FC}(S)=\{\Delta _S,\nabla _S\}\cong {\cal L}_2$.

${\rm Con}(R)=\{\Delta _R,\rho _1,\rho _2,\rho _3,\rho _4,\rho _5,\rho _6,\nabla _R\}$, where $\rho _1=eq(\{0,a,b,c,x\},\{y,1\})$, $\rho _2=eq(\{0,a,b,c,x,y\},\{1\})$, $\rho _3=eq(\{0,a,b,c,x\},\{y\},\{1\})$, $\rho _4=eq(\{0\},\{a\},\{b\},\{c\}\},\{x,y,1\})$, $\rho _5=eq(\{0\},\{a\},\{b\},\{c\}\},\{x\},\{y,1\})$ and $\rho _6=eq(\{0\},\{a\},\{b\},\{c\}\},\{x,y\},\{1\})$ and with the lattice structure represented above, hence ${\cal B}({\rm Con}(R))={\rm Con}(R)\cong {\cal L}_2^3$, in which $\neg \, \rho _1=\rho _6$, $\neg \, \rho _2=\rho _5$ and $\neg \, \rho _3=\rho _4$. Now let us notice that $(y,b)\in \rho _1\circ \rho _6$, but $(y,b)\notin \rho _6\circ \rho _1$, $(1,b)\in \rho _2\circ \rho _5$, but $(1,b)\notin \rho _5\circ \rho _2$, and $(1,b)\in \rho _3\circ \rho _4$, but $(1,b)\notin \rho _4\circ \rho _3$, hence $\rho _1\circ \rho _6\neq \rho _6\circ \rho _1$, $\rho _2\circ \rho _5\neq \rho _5\circ \rho _2$ and $\rho _3\circ \rho _4\neq \rho _4\circ \rho _3$, therefore ${\rm FC}(R)=\{\Delta _R,\nabla _R\}\cong {\cal L}_2$.

${\rm Con}(T)=\{\Delta _T,\tau _1,\tau _2,\tau _3,\tau _4,\tau _5,\tau _6,\nabla _T\}$, where $\tau _1=eq(\{0,z,a,b,c,x\},\{1\})$, $\tau _2=eq(\{0\},\{z,a,b,c,x,1\})$, $\tau _3=eq(\{0\},\{z,a,b,c,x\},\{1\})$, $\tau _4=eq(\{0,z\},\{a\},\{b\},\{c\}\},\{x,1\})$, $\tau _5=eq(\{0,z\},\{a\},\{b\},\{c\}\},\{x\},\{1\})$ and $\tau _6=eq(\{0\},\{z\},\{a\},\{b\},\{c\}\},\{x,1\})$ and with the Hasse diagram above, hence ${\cal B}({\rm Con}(T))={\rm Con}(T)\cong {\cal L}_2^3$, in which $\neg \, \tau _1=\tau _6$, $\neg \, \tau _2=\tau _5$ and $\neg \, \tau _3=\tau _4$. Now we may notice that $(1,0)\in \tau _1\circ \tau _6$, but $(1,0)\notin \tau _6\circ \tau _1$, $(1,0)\in \tau _5\circ \tau _2$, but $(1,0)\notin \tau _2\circ \tau _5$, and $(1,b)\in \tau _3\circ \tau _4$, but $(1,b)\notin \tau _4\circ \tau _3$, hence $\tau _1\circ \tau _6\neq \tau _6\circ \tau _1$, $\tau _2\circ \tau _5\neq \tau _5\circ \tau _2$ and $\tau _3\circ \tau _4\neq \tau _4\circ \tau _3$, therefore ${\rm FC}(T)=\{\Delta _T,\nabla _T\}\cong {\cal L}_2$.

${\rm Con}(E)=\{\Delta _E,\varepsilon ,\nabla _E\}\cong {\cal L}_3$, where $\varepsilon =eq(\{0\},\{a\},\{b,d\},\{c\},\{1\})$, hence ${\cal B}({\rm Con}(E))=\{\Delta _E,\nabla _E\}\cong {\cal L}_2$, thus ${\rm FC}(E)={\cal B}({\rm Con}(E))=\{\Delta _E,\nabla _E\}\cong {\cal L}_2$.

The examples above, together with Remark \ref{oremarca} and Lemma \ref{lemma3.3}, provide us with many more meaningful examples: for instance, the finite non--distributive lattice $R^2$ has Boolean congruences that are not factor congruences, and ${\rm FC}(R^2)\supsetneq \{\Delta _{R^2},\nabla _{R^2}\}$, because, by the above, ${\rm Con}(R^2)={\cal B}({\rm Con}(R^2))\cong {\cal L}_2^6$ and ${\rm FC}(R^2)\cong {\cal L}_2^2$; an example of a finite distributive lattice with these properties is ${\cal L}_2\times {\cal L}_3$ (see Example \ref{ex}). Another finite non--distributive lattice with these properties, but which has, furthermore, congruences which are not Boolean is, for instance, $T\times E$, which has ${\rm Con}(T\times E)\cong {\cal L}_2^3\times {\cal L}_3$, ${\cal B}({\rm Con}(T\times E))\cong {\cal L}_2^4$ and ${\rm FC}(T\times E)\cong {\cal L}_2^2$.\label{ex}\end{example}

\begin{remark} Let $A$ and $B$ be congruence--distributive algebras. Then:\begin{enumerate}
\item ${\rm Con}(A)\cong {\rm Con}(B)\ \Rightarrow \ {\cal B}({\rm Con}(A))\cong {\cal B}({\rm Con}(B))$;
\item ${\cal B}({\rm Con}(A))\cong {\cal B}({\rm Con}(B))\ \nRightarrow \ {\rm FC}(A)\cong {\rm FC}(B)$;
\item ${\rm Con}(A)\cong {\rm Con}(B)\ \nRightarrow \ {\rm FC}(A)\cong {\rm FC}(B)$;
\item ${\rm Con}(A)={\cal B}({\rm Con}(A))\cong {\rm Con}(B)={\cal B}({\rm Con}(B))\ \nRightarrow \ {\rm FC}(A)\cong {\rm FC}(B)$.\end{enumerate}

Indeed, the first statement is trivial, while Example \ref{ex} provides us with many counter--examples for the other implications; for instance, ${\rm Con}({\cal L}_2^3)={\cal B}({\rm Con}({\cal L}_2^3))\cong {\rm Con}({\cal L}_2\times {\cal L}_3)={\cal B}({\rm Con}({\cal L}_2\times {\cal L}_3))\cong {\rm Con}(R)={\cal B}({\rm Con}(R))\cong {\cal L}_2^3$, while ${\rm FC}({\cal L}_2^3)\cong {\cal L}_2^3$, ${\rm FC}({\cal L}_2\times {\cal L}_3)\cong {\cal L}_2^2$ and ${\rm FC}(R)\cong {\cal L}_2$.\end{remark}

Throughout the rest of this section, $A$ shall be a congruence--distributive algebra and $\theta \in {\rm Con}(A)$, unless mentioned otherwise. The properties on CBLP in the following remarks are known from \cite{cblp}, but we are including them in these results for the sake of completeness. We shall only prove the statements on FCLP; note that those on CBLP are easily derivable in the same manner as the ones on FCLP.

\begin{remark}\begin{itemize}
\item $\Delta _{A}$ and $\nabla _{A}$ have CBLP and FCLP. Indeed, $p_{\Delta _{A}}:A\rightarrow A/\Delta _{A}$ is an isomorphism, hence, according to Remark \ref{transpcong}, ${\rm FC}(A/\Delta _{A})=\{p_{\Delta _{A}}(\alpha )\ |\ \alpha \in {\rm FC}(A)\}=\{\alpha /\Delta _{A}\ |\ \alpha \in {\rm FC}(A)\}$, and, for all $\alpha \in {\rm FC}(A)$, ${\rm FC}(\Delta _{A})(\alpha )=u_{\Delta _{A}}(\alpha )=(\alpha \vee \Delta _{A})/\Delta _{A}=\alpha /\Delta _{A}$, thus ${\rm FC}(\Delta _{A})$ is surjective, that is $\Delta _{A}$ has FCLP. ${\rm Con}(A/\nabla _{A})={\cal B}({\rm Con}(A/\nabla _{A}))={\rm FC}(A/\nabla _{A})=\{\nabla _{A}/\nabla _{A}\}=\{\Delta _{(A/\nabla _{A})}\}=\{\nabla _{(A/\nabla _{A})}\}\cong {\cal L}_1$, thus ${\rm FC}(\nabla _{A}):{\rm FC}(A)\rightarrow {\rm FC}(A/\nabla _{A})$ is clearly surjective, that is $\nabla _{A}$ has FCLP.
\item By the previous statement, if ${\rm Con}(A)=\{\Delta _{A},\nabla _{A}\}$, then $A$ has CBLP and FCLP. Consequently, the trivial algebra has CBLP and FCLP.
\item If ${\rm FC}(A/\theta )=\{\Delta _{A/\theta },\nabla _{A/\theta}\}$ ($\cong {\cal L}_1$ or $\cong {\cal L}_2$), then $\theta $ has FCLP, because in this case the Boolean morphism ${\rm FC}(\theta ):{\rm FC}(A)\rightarrow {\rm FC}(A/\theta )$ is clearly surjective. This is a generalization of the case $\theta =\nabla _A$.
\item If ${\cal B}({\rm Con}(A/\theta ))=\{\Delta _{A/\theta },\nabla _{A/\theta}\}$, then $\theta $ has CBLP and FCLP, because in this case the Boolean morphism ${\cal B}(u_{\theta })$ is clearly surjective, and we also have ${\rm FC}(A/\theta )=\{\Delta _{A/\theta },\nabla _{A/\theta}\}$, so we can apply the previous statement.
\item If ${\rm FC}(A)=\{\Delta _{A},\nabla _{A}\}$, then: $\theta $ has FCLP iff ${\rm FC}(A/\theta )=\{\Delta _{A/\theta },\nabla _{A/\theta}\}$, because ${\rm FC}(\theta )(\{\Delta _{A},\nabla _{A}\})=\{(\Delta _{A}\vee \theta )/\theta ,(\nabla _{A}\vee \theta )/\theta \}=\{\theta /\theta ,\nabla _{A}/\theta \}=\{\Delta _{A/\theta },\nabla _{A/\theta}\}\subseteq {\rm FC}(A/\theta )$.
\item If ${\cal B}({\rm Con}(A))=\{\Delta _{A},\nabla _{A}\}$, then we also have ${\rm FC}(A)=\{\Delta _{A},\nabla _{A}\}$, so, just as above: $\theta $ has CBLP iff ${\cal B}({\rm Con}(A/\theta ))=\{\Delta _{A/\theta },\nabla _{A/\theta}\}$, and $\theta $ has FCLP iff ${\rm FC}(A/\theta )=\{\Delta _{A/\theta },\nabla _{A/\theta}\}$, so, if $\theta $ has CBLP, then $\theta $ has FCLP.\end{itemize}\label{r4.11}\end{remark}

\begin{remark}\begin{itemize}
\item If ${\cal B}({\rm Con}(A))={\rm FC}(A)$ and $\theta $ has CBLP, then: $\theta $ has FCLP and ${\cal B}({\rm Con}(A/\theta ))={\rm FC}(A/\theta )$. Indeed, if ${\cal B}({\rm Con}(A))={\rm FC}(A)$, then ${\rm FC}(\theta )=u_{\theta }\mid _{\textstyle {\rm FC}(A)}=u_{\theta }\mid _{\textstyle {\cal B}({\rm Con}(A))}={\cal B}(u_{\theta })$, and ${\rm FC}(A/\theta )\subseteq {\cal B}({\rm Con}(A/\theta ))$, thus, if $\theta $ has CBLP, so that ${\cal B}(u_{\theta })$ is surjective, then ${\cal B}({\rm Con}(A/\theta ))={\cal B}(u_{\theta })({\cal B}({\rm Con}(A)))={\rm FC}(\theta )({\rm FC}(A))\subseteq {\rm FC}(A/\theta )\subseteq {\cal B}({\rm Con}(A/\theta ))$, hence ${\rm FC}(\theta )({\rm FC}(A))={\rm FC}(A/\theta )$, that is ${\rm FC}(\theta )$ is surjective, which means that $\theta $ has FCLP.
\item If ${\cal B}({\rm Con}(A))={\rm FC}(A)$ and ${\cal B}({\rm Con}(A/\theta ))={\rm FC}(A/\theta )$, then: $\theta $ has CBLP iff $\theta $ has FCLP. Indeed, if ${\cal B}({\rm Con}(A))={\rm FC}(A)$, then, as above, ${\cal B}(u_{\theta })={\rm FC}(\theta )$, hence ${\cal B}(u_{\theta })({\cal B}({\rm Con}(A)))={\rm FC}(\theta )({\rm FC}(A))$, so, if, moreover, ${\cal B}({\rm Con}(A/\theta ))={\rm FC}(A/\theta )$, then we have the following equivalence: ${\cal B}(u_{\theta })({\cal B}({\rm Con}(A)))={\cal B}({\rm Con}(A/\theta ))$ iff ${\rm FC}(\theta )({\rm FC}(A))={\rm FC}(A/\theta )$, which means that: ${\cal B}(u_{\theta })$ is surjective iff ${\rm FC}(\theta )$ is surjective, that is: $\theta $ has CBLP iff $\theta $ has FCLP. 
\item If ${\cal B}({\rm Con}(A))={\rm FC}(A)$ and $A$ has CBLP, then: $A$ has FCLP and ${\cal B}({\rm Con}(A/\phi ))={\rm FC}(A/\phi )$ for all $\phi \in {\rm Con}(A)$. This follows from the first statement in this remark.
\item If ${\cal B}({\rm Con}(A/\phi ))={\rm FC}(A/\phi )$ for all $\phi \in {\rm Con}(A)$, then: $A$ has CBLP iff $A$ has FCLP. This follows from the second statement in this remark and the fact that $A/\Delta _A\cong A$, hence, by Remark \ref{transpcong}, ${\cal B}({\rm Con}(A/\Delta _A))={\rm FC}(A/\Delta _A)$ iff ${\cal B}({\rm Con}(A))={\rm FC}(A)$.
\item If ${\cal B}({\rm Con}(A))={\rm Con}(A)$, then $A$ has CBLP and, for all $\phi \in {\rm Con}(A)$, ${\cal B}({\rm Con}(A/\phi ))={\rm Con}(A/\phi )$. This is known from \cite{cblp}, but also follows easily from the fact that, in this case, for all $\phi \in {\rm Con}(A)$, ${\cal B}(u_{\theta })=u_{\theta }$, which is surjective, according to Remark \ref{4***}.
\item If ${\rm FC}(A)={\rm Con}(A)$, then $A$ has CBLP and FCLP and, for all $\phi \in {\rm Con}(A)$, ${\rm FC}(A/\phi )={\cal B}({\rm Con}(A/\phi ))={\rm Con}(A/\phi )$. Indeed, if ${\rm FC}(A)={\rm Con}(A)$, then, since ${\rm FC}(A)\subseteq {\cal B}({\rm Con}(A))\subseteq {\rm Con}(A)$, it follows that ${\rm FC}(A)={\cal B}({\rm Con}(A))={\rm Con}(A)$, hence, by the previous statement and the first statement in this remark, $A$ has CBLP, therefore $A$ has FCLP, and, for all $\phi \in {\rm Con}(A)$, ${\rm FC}(A/\phi )={\cal B}({\rm Con}(A/\phi ))={\rm Con}(A/\phi )$.
\item If $[\theta )\subseteq {\cal B}({\rm Con}(A))$, then each $\phi \in [\theta )$ has CBLP and fulfills ${\cal B}({\rm Con}(A/\phi ))={\rm Con}(A/\phi )$. This is known from \cite{cblp}, but can also be derived just as the part on FCLP in the next statement.
\item If $[\theta )\subseteq {\rm FC}(A)$, then each $\phi \in [\theta )$ has CBLP and FCLP and fulfills ${\rm FC}(A/\phi )={\cal B}({\rm Con}(A/\phi ))={\rm Con}(A/\phi )$. Indeed, if $[\theta )\subseteq {\rm FC}(A)\subseteq {\cal B}({\rm Con}(A))$ and $\phi \in [\theta )$, then, by the previous statement, $\phi $ has CBLP and ${\cal B}({\rm Con}(A/\phi ))={\rm Con}(A/\phi )$. Furthermore, we have the following: for each $\gamma \in {\rm FC}(A/\phi )\subseteq {\rm Con}(A/\phi )$, there exists an $\alpha \in [\phi )\subseteq [\theta )\subseteq {\rm FC}(A)$ such that $\gamma =\alpha /\phi =(\alpha \vee \phi )/\phi =u_{\phi }(\alpha )={\rm FC}(\phi )(\alpha )$, thus ${\rm FC}(\phi )$ is surjective, that is $\phi $ has FCLP and, furthermore, ${\rm FC}(A/\phi )\subseteq {\rm Con}(A/\phi )\subseteq {\rm FC}(\phi )( {\rm FC}(A))={\rm FC}(A/\phi )$, hence ${\rm FC}(A/\phi )={\rm Con}(A/\phi )$, thus ${\rm FC}(A/\phi )={\cal B}({\rm Con}(A/\phi ))={\rm Con}(A/\phi )$.\end{itemize}\label{r4.12}\end{remark}

\begin{example} Let us determine, for the lattices in Example \ref{ex}, as well as each of their congruences, whether they have CBLP or FCLP. We shall use the calculations in Example \ref{ex} and the first statement in Remark \ref{r4.11}.

${\cal L}_2$, ${\cal L}_3$, ${\cal L}_2^2$, ${\cal L}_2^3$ and ${\cal L}_2\times {\cal L}_3$ are bounded distributive lattices, hence they have CBLP by Lemma \ref{cblpabvd}. ${\cal L}_2$, ${\cal L}_2^2$ and ${\cal L}_2^3$ are Boolean algebras, hence they have FCLP by Proposition \ref{fclpabvd}.

${\cal L}_3/\phi \cong {\cal L}_3/\psi \cong {\cal L}_2$, which has ${\rm FC}({\cal L}_2)={\cal B}({\rm Con}({\cal L}_2))={\rm Con}({\cal L}_2)\cong {\cal L}_2$, thus, by Remark \ref{transpcong}: ${\rm FC}({\cal L}_3/\phi )={\cal B}({\rm Con}({\cal L}_3/\phi ))={\rm Con}({\cal L}_3/\phi )\cong {\rm FC}({\cal L}_3/\psi )={\cal B}({\rm Con}({\cal L}_3/\psi ))={\rm Con}({\cal L}_3/\psi )\cong {\cal L}_2$, therefore ${\rm FC}({\cal L}_3/\phi )=\{\Delta _{{\cal L}_3/\phi },\nabla _{{\cal L}_3/\phi }\}$ and ${\rm FC}({\cal L}_3/\psi )=\{\Delta _{{\cal L}_3/\psi },\nabla _{{\cal L}_3/\psi }\}$, hence $\phi $ and $\psi $ have FCLP by Remark \ref{r4.11}. Therefore ${\cal L}_3$ has FCLP. So ${\cal L}_2$ and ${\cal L}_3$ have FCLP, hence ${\cal L}_2\times {\cal L}_3$ has FCLP by Proposition \ref{proposition4.5}.

${\rm Con}({\cal D})=\{\Delta _{\cal D},\nabla _{\cal D}\}$, thus ${\cal D}$ has CBLP and FCLP by Remark \ref{r4.11}. ${\cal P}/\alpha \cong {\cal P}/\beta \cong {\cal L}_2$, hence, just as above, it follows that $\alpha $ and $\beta $ have CBLP and FCLP. ${\cal P}/\alpha \cong {\cal L}_2^2$, thus, by Remark \ref{transpcong}, ${\rm FC}({\cal P}/\gamma )={\cal B}({\rm Con}({\cal P}/\gamma ))={\rm Con}({\cal P}/\gamma )\cong {\cal L}_2^2$, therefore ${\cal B}(u_{\gamma }):{\cal B}({\rm Con}({\cal P}))\cong {\cal L}_2\rightarrow {\cal B}({\rm Con}({\cal P}/\gamma ))\cong {\cal L}_2^2$ and ${\rm FC}(\gamma ):{\rm FC}({\cal P})\cong {\cal L}_2\rightarrow {\rm FC}({\cal P}/\gamma )\cong {\cal L}_2^2$, hence neither of these Boolean morphisms is surjective, thus $\gamma $ has neither CBLP, nor FCLP. Therefore ${\cal P}$ has neither CBLP, nor FCLP. The fact that ${\cal D}$ has CBLP, while ${\cal P}$ does not have CBLP was known from \cite{cblp}, but we have shown it here, as well, for the sake of completeness.

${\cal B}({\rm Con}(S))={\rm Con}(S)$, ${\cal B}({\rm Con}(R))={\rm Con}(R)$ and ${\cal B}({\rm Con}(T))={\rm Con}(T)$, thus $S$, $R$ and $T$ have CBLP by Remark \ref{r4.12}. $S/\sigma _1\cong {\cal D}$, which has ${\rm FC}({\cal D})=\{\Delta _{\cal D},\nabla _{\cal D}\}$, thus, by Remark \ref{transpcong}, it follows that ${\rm FC}(S/\sigma _1)=\{\Delta _{S/\sigma _1},\nabla _{S/\sigma _1}\}$, therefore $\sigma _1$ has FCLP by Remark \ref{r4.11}. $S/\sigma _2\cong {\cal L}_2$, thus, as above, it follows that $\sigma _2$ has FCLP. Therefore $S$ has FCLP, as well. $R/\rho _1\cong R/\rho _2\cong {\cal L}_2$, thus, as above, $\rho _1$ and $\rho _2$ have FCLP. $R/\rho _3\cong {\cal L}_3$, thus, by Remark  \ref{transpcong}, ${\rm FC}(R/\rho _3)\cong {\rm FC}({\cal L}_3)\cong {\cal L}_2$, so ${\rm FC}(R/\rho _3)=\{\Delta _{R/\rho _3},\nabla _{R/\rho _3}\}$, hence $\rho _3$ has FCLP by Remark \ref{r4.11}. $R/\rho _4\cong {\cal D}$, hence, as above, it follows that $\rho _4$ has FCLP. $R/\rho _5\cong R/\rho _6\cong S$, thus, by Remark \ref{transpcong}, ${\rm FC}(R/\rho _5)\cong {\rm FC}(R/\rho _6)\cong {\rm FC}(S)\cong {\cal L}_2$, so ${\rm FC}(R/\rho _5)=\{\Delta _{R/\rho _5},\nabla _{R/\rho _5}\}$ and ${\rm FC}(R/\rho _6)=\{\Delta _{R/\rho _6},\nabla _{R/\rho _6}\}$, hence $\rho _5$ and $\rho _6$ have FCLP, by Remark \ref{r4.11}. Therefore $R$ has FCLP, too. $T/\tau _1\cong T/\tau _2\cong {\cal L}_2$, $T/\tau _3\cong {\cal L}_3$, $T/\tau _4\cong {\cal D}$ and $T/\tau _5\cong S$, hence, as above, it follows that $\tau _1$, $\tau _2$, $\tau _3$, $\tau _4$ and $\tau _5$ have FCLP. $T/\tau _6$ is isomorphic to the dual of $S$, hence, by Remark \ref{transpcong}, ${\rm FC}(T/\tau _6)\cong {\rm FC}(S)\cong {\cal L}_2$, so ${\rm FC}(T/\tau _6)=\{\Delta _{T/\tau _6},\nabla _{T/\tau _6}\}$, thus $\tau _6$ has FCLP by Remark \ref{r4.11}. Therefore $T$ has FCLP, too.

$E/\varepsilon \cong {\cal D}$, which has ${\rm FC}({\cal D})={\cal B}({\rm Con}({\cal D}))={\rm Con}({\cal D})=\{\Delta _{\cal D},\nabla _{\cal D}\}$, thus, by Remark \ref{transpcong}, it follows that ${\rm FC}(E/\varepsilon )={\cal B}({\rm Con}(E/\varepsilon ))={\rm Con}(E/\varepsilon )=\{\Delta _{E/\varepsilon },\nabla _{E/\varepsilon }\}$, therefore $\varepsilon $ has CBLP and FCLP by Remark \ref{r4.11}. Hence $E$ has CBLP and FCLP.\label{ex4.14}\end{example}

\begin{remark}\begin{itemize}
\item From Corollary \ref{nicecor} and Example \ref{ex4.14}, it follows that no ordinal sum of lattices in which ${\cal P}$ appears has CBLP or FCLP. This produces both finite and infinite examples of non--modular bounded lattices which have neither CBLP, nor FCLP. Notice, however, that, according to Example \ref{ex4.14}, the non--modular bounded lattice $E$ has both CBLP and FCLP.
\item See also Example \ref{nd} below, featuring a modular bounded lattice without FCLP. Concerning the issue of how to seek for modular bounded lattices without CBLP, note that, by Remark \ref{congrsl2}, the fact that ${\cal B}({\rm Con}({\cal D}))={\rm Con}({\cal D})$, proven in Example \ref{ex}, and Remark \ref{transpcong}, if $n\in \N ^*$, $L_1,\ldots ,L_n$ are finite lattices which are either distributive or isomorphic to ${\cal D}$, so that ${\cal B}({\rm Con}(L_i)={\rm Con}(L_i)$ for all $i\in \overline{1,n}$, and if $L\cong L_1\dotplus \ldots \dotplus L_n$ and $M\cong L_1\times \ldots \times L_n$, then ${\cal B}({\rm Con}(L)={\rm Con}(L)$ and ${\cal B}({\rm Con}(M)={\rm Con}(M)$, thus $L$ and $M$ have CBLP by Remark \ref{r4.12} (see also \cite{cblp}). This holds for infinite ordinal sums and infinite direct products, as well.\end{itemize}\label{z}\end{remark}

\begin{example} Now let us see that the implication in Corollary \ref{compar1} does not hold in bounded non--distributive lattices either. Let us consider the following bounded modular non--distributive lattice: $X={\cal L}_2^2\dotplus {\cal D}$:\vspace*{-18pt}

\begin{center}
\begin{tabular}{ccc}
\begin{picture}(40,130)(0,0)
\put(30,65){\line(0,1){40}}
\put(30,105){\line(1,-1){20}}
\put(30,105){\line(-1,-1){20}}
\put(30,25){\line(1,1){20}}
\put(30,25){\line(-1,1){20}}
\put(10,45){\line(1,1){40}}
\put(50,45){\line(-1,1){40}}
\put(30,25){\circle*{3}}

\put(10,45){\circle*{3}}
\put(50,45){\circle*{3}}
\put(10,85){\circle*{3}}
\put(50,85){\circle*{3}}
\put(30,65){\circle*{3}}
\put(30,85){\circle*{3}}
\put(30,105){\circle*{3}}
\put(2,42){$p$}
\put(2,82){$s$}
\put(53,42){$q$}
\put(53,82){$u$}
\put(35,62){$r$}
\put(33,82){$t$}
\put(28,15){$0$}
\put(28,108){$1$}
\put(2,0){$X={\cal L}_2^2\dotplus {\cal D}$}
\end{picture}
&\hspace*{60pt}
\begin{picture}(60,130)(0,0)
\put(30,30){\line(0,1){20}}
\put(10,50){\line(0,1){20}}
\put(50,50){\line(0,1){20}}
\put(30,70){\line(0,1){20}}
\put(30,30){\line(1,1){20}}
\put(30,30){\line(-1,1){20}}
\put(30,50){\line(1,1){20}}
\put(30,50){\line(-1,1){20}}
\put(30,70){\line(1,-1){20}}
\put(30,70){\line(-1,-1){20}}
\put(30,90){\line(1,-1){20}}
\put(30,90){\line(-1,-1){20}}
\put(30,30){\circle*{3}}
\put(30,50){\circle*{3}}
\put(10,50){\circle*{3}}
\put(50,50){\circle*{3}}
\put(10,70){\circle*{3}}
\put(50,70){\circle*{3}}
\put(30,70){\circle*{3}}
\put(30,90){\circle*{3}}
\put(26,20){$\Delta _X$}

\put(26,93){$\nabla _X$}
\put(-1,48){$\xi _5$}
\put(-1,70){$\xi _1$}
\put(33,47){$\xi _4$}
\put(33,70){$\xi _3$}
\put(53,48){$\xi _6$}
\put(53,70){$\xi _2$}
\put(-20,3){${\rm Con}(X)={\cal B}({\rm Con}(X))$}\end{picture}
&\hspace*{50pt}
\begin{picture}(20,130)(0,0)
\put(6,20){$\Delta _X$}
\put(6,53){$\nabla _X$}
\put(10,30){\line(0,1){20}}
\put(10,30){\circle*{3}}
\put(10,50){\circle*{3}}
\put(0,3){${\rm FC}(X)$}\end{picture}\end{tabular}\end{center}\vspace*{-1pt}

Example \ref{ex} and Remark \ref{congrsl2} show that ${\rm Con}(X)={\cal B}({\rm Con}(X))=\{\Delta _X,\xi _1,\xi _2,\xi _3,\xi _4,\xi _5,\xi _6,\nabla _X\}\cong {\cal L}_2^3$, where $\xi _1=eq(\{0,q\},\{p,r,s,t,u,1\})$, $\xi _2=eq(\{0,p\},\{q,r,s,t,u,1\})$, $\xi _3=eq(\{0,p,q,r\},\{s\},\{t\},\{u\},\{1\})$, $\xi _4=eq(\{0\},\{p\},\{q\},\{r,s,t,u,1\})$, $\xi _5=eq(\{0,q\},\{p,r\},\{s\},\{t\},\{u\},\{1\})$ and $\xi _6=eq(\{0,p\},\{q,r\},\{s\},\{t\},$\linebreak $\{u\},\{1\})$. Thus $X$ has CBLP by Remark \ref{r4.12}. In the Boolean algebra ${\cal B}({\rm Con}(X))$, $\neg \, \xi _1=\xi _6$, $\neg \, \xi _2=\xi _5$ and $\neg \, \xi _3=\xi _4$. Now, if we notice that $(1,0)\in \xi _6\circ \xi _1,\xi _5\circ \xi _2,\xi _3\circ \xi _4$ and $(1,0)\notin \xi _1\circ \xi _6,\xi _2\circ \xi _5,\xi _4\circ \xi _3$, and thus $\xi _6\circ \xi _1\neq \xi _1\circ \xi _6$, $\xi _5\circ \xi _2\neq \xi _2\circ \xi _5$ and $\xi _3\circ \xi _4\neq \xi _4\circ \xi _3$, then we conclude that ${\rm FC}(X)=\{\Delta _X,\nabla _X\}\cong {\cal L}_2$. Now it suffices to observe that $X/\xi _4\cong {\cal L}_2^2$, which is a finite Boolean algebra, hence it has ${\rm FC}({\cal L}_2^2)={\cal B}({\rm Con}({\cal L}_2^2))={\rm Con}({\cal L}_2^2)\cong {\cal L}_2^2$, therefore we have: ${\rm FC}(\xi _4):{\rm FC}(X)\cong {\cal L}_2\rightarrow {\rm FC}(X/\xi _4)\cong {\cal L}_2^2$, thus ${\rm FC}(\xi _4)$ is not surjective, that is $\xi _4$ does not have FCLP, hence $X$ does not have FCLP.\label{nd}\end{example}

\begin{example} Now let us see an example of a lattice with FCLP and without CBLP. Let $H$ be the following non--modular bounded lattice:\vspace*{-6pt}

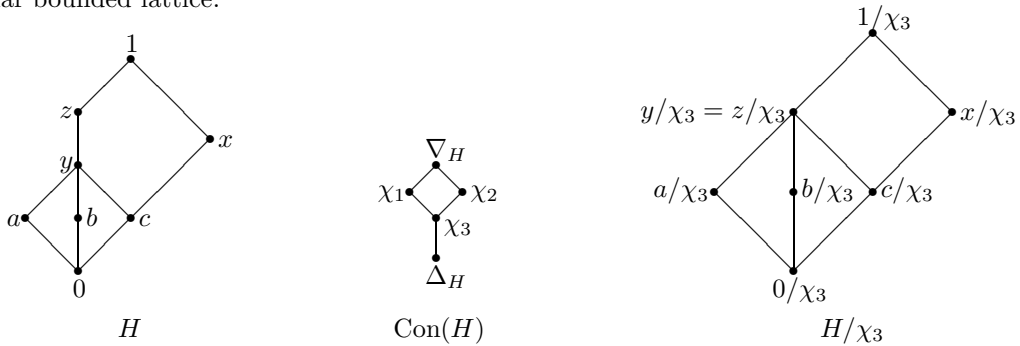
\begin{center}\begin{tabular}{ccc}
\begin{picture}(40,120)(0,0)
\put(30,25){\line(0,1){60}}
\put(30,25){\line(1,1){50}}
\put(30,25){\line(-1,1){20}}
\put(30,65){\line(1,-1){20}}
\put(30,85){\line(1,1){20}}
\put(30,65){\line(-1,-1){20}}
\put(50,105){\line(1,-1){30}}
\put(3,42){$a$}

\put(33,42){$b$}

\put(53,42){$c$}

\put(23,64){$y$}

\put(23,83){$z$}

\put(83,72){$x$}

\put(30,25){\circle*{3}}
\put(10,45){\circle*{3}}
\put(30,45){\circle*{3}}
\put(50,45){\circle*{3}}
\put(30,65){\circle*{3}}
\put(30,85){\circle*{3}}
\put(50,105){\circle*{3}}
\put(80,75){\circle*{3}}
\put(28,15){$0$}
\put(48,108){$1$}
\put(45,0){$H$}
\end{picture}
&\hspace{90pt}
\begin{picture}(40,120)(0,0)
\put(20,45){\line(1,1){10}}
\put(20,45){\line(-1,1){10}}
\put(20,65){\line(1,-1){10}}
\put(20,65){\line(-1,-1){10}}
\put(20,45){\line(0,-1){15}}
\put(16,20){$\Delta _H$}
\put(16,68){$\nabla _H$}
\put(20,30){\circle*{3}}
\put(-2,53){$\chi _1$}
\put(33,53){$\chi _2$}
\put(23,40){$\chi _3$}
\put(20,45){\circle*{3}}
\put(10,55){\circle*{3}}
\put(30,55){\circle*{3}}
\put(20,65){\circle*{3}}
\put(4,0){${\rm Con}(H)$}
\end{picture}
&\hspace{70pt}
\begin{picture}(120,120)(0,0)
\put(30,25){\line(0,1){60}}
\put(30,25){\line(1,1){60}}
\put(30,25){\line(-1,1){30}}
\put(0,55){\line(1,1){60}}
\put(60,55){\line(-1,1){30}}
\put(60,115){\line(1,-1){30}}
\put(30,25){\circle*{3}}
\put(0,55){\circle*{3}}
\put(30,55){\circle*{3}}
\put(60,55){\circle*{3}}
\put(30,85){\circle*{3}}
\put(60,115){\circle*{3}}
\put(90,85){\circle*{3}}
\put(22,15){$0/\chi _3$}

\put(54,118){$1/\chi _3$}
\put(-23,53){$a/\chi _3$}

\put(33,53){$b/\chi _3$}

\put(63,53){$c/\chi _3$}

\put(-28,83){$y/\chi _3=z/\chi _3$}

\put(93,82){$x/\chi _3$}

\put(40,0){$H/\chi _3$}
\end{picture}\end{tabular}\end{center}

It is easy to obtain, by using Remark \ref{congrsl1} and the calculations in Example \ref{ex}, that ${\rm Con}(H)=\{\Delta _H,\chi _1,\chi _2,\chi _3,$\linebreak $\nabla _H\}$, with the lattice structure represented above, where $\chi _1=eq(\{0,a,b,c,y,z\},\{x,1\})$, $\chi _2=eq(\{0\},\{a\},\{b\},$\linebreak $\{c,x\},\{y,z,1\})$ and $\chi _3=eq(\{0\},\{a\},\{b\},\{c\},\{x\},\{y,z\},\{1\})$. Thus ${\rm FC}(H)={\cal B}({\rm Con}(H))=\{\Delta _H,\nabla _H\}\cong {\cal L}_2$. $H/\chi _1\cong {\cal L}_2$ and $H/\chi _2\cong {\cal D}$, so, just as in Example \ref{ex4.14}, $\chi _1$ and $\chi _2$ have CBLP and FCLP.

$H/\chi _3$ has the Hasse diagram above. Remark \ref{congrsl1} and Example \ref{ex} make it easy to obtain that ${\rm Con}(H/\chi _3)={\cal B}({\rm Con}(H/\chi _3))=\{\Delta _{H/\chi _3},\nu ,\pi ,\nabla _{H/\chi _3}\}\cong {\cal L}_2^2$, with $\nu =eq(\{0/\chi _3,a/\chi _3,b/\chi _3,c/\chi _3,y/\chi _3\},\{x/\chi _3,1/\chi _3\})$ and $\pi =eq(\{0/\chi _3\},\{a/\chi _3\},\{b/\chi _3\},\{c/\chi _3,x/\chi _3\},\{y/\chi _3,1/\chi _3\})$, so with $\pi =\neg \, \nu $. Since $(0/\chi _3,1/\chi _3)\in \pi \circ \nu $, but $(0/\chi _3,1/\chi _3)\notin \nu \circ \pi $, it follows that $\pi \circ \nu \neq \nu \circ \pi $, hence ${\rm FC}(H/\chi _3)=\{\Delta _{H/\chi _3},\nabla _{H/\chi _3}\}\cong {\cal L}_2$, so $\chi _3$ has FCLP by Remark \ref{r4.11}. But ${\cal B}(u_{H/\chi _3}):{\cal B}({\rm Con}(H))\cong {\cal L}_2\rightarrow {\cal B}({\rm Con}(H/\chi _3))\cong {\cal L}_2^2$, thus ${\cal B}(u_{H/\chi _3})$ is not surjective, that is $\chi _3$ does not have CBLP. Therefore $H$ has FCLP, but it does not have CBLP.\label{exc}\end{example}

\begin{corollary} FCLP does not imply CBLP.\label{compar2}\end{corollary}

\begin{proof} By Example \ref{exc}.\end{proof}

\begin{proposition}\begin{enumerate}
\item\label{nimpl1} If $A$ has FCLP, then its subalgebras do not necessarily have FCLP. The same goes for CBLP instead of FCLP.
\item\label{nimpl2} The fact that all proper subalgebras of $A$ have FCLP does not imply that $A$ has FCLP.
\item\label{nimpl3} The fact that all proper quotient algebras of $A$ have FCLP does not imply that $A$ has FCLP.\end{enumerate}\label{nimpl}\end{proposition}

\begin{proof} (\ref{nimpl1}) By Example \ref{ex4.14}, $E$ has CBLP and FCLP, although it has sublattices isomorphic to ${\cal P}$, which has neither CBLP, nor FCLP. The fact on CBLP was known from \cite{cblp}.

\noindent (\ref{nimpl2}), (\ref{nimpl3}) Let $L={\cal L}_2\dotplus {\cal L}_2^2$. As pointed out in Remark \ref{clllps}, $L$ does not have FCLP. Every proper subalgebra and every proper quotient algebra of $L$ is isomorphic to one of the Boolean algebras ${\cal L}_2$ and ${\cal L}_2^2$, which have FCLP by 
Proposition \ref{fclpabvd}. For (\ref{nimpl3}) we can provide a non--distributive example, too: the lattice $X$ in Example \ref{nd} does not have FCLP, but each of its proper quotient algebras is isomorphic to ${\cal L}_2$, ${\cal L}_2^2$, ${\cal D}$ or the dual of $S$ from Example \ref{ex}, and all these lattices have FCLP, by Example \ref{ex4.14} and the fact that FCLP is self--dual.\end{proof}

\begin{proposition}\begin{enumerate}
\item\label{spec1} Any maximal congruence of $A$ has FCLP and CBLP.
\item\label{spec2} Any prime congruence of $A$ has FCLP and CBLP.\end{enumerate}\label{spec}\end{proposition}

\begin{proof} The statements on CBLP are known from \cite{cblp}, but also follow from the next arguments.

\noindent (\ref{spec1}) Let $\theta \in {\rm Max}(A)$. Then $[\theta )=\{\theta ,\nabla _A\}$, thus ${\rm Con}(A/\theta )=\{\theta /\theta,\nabla _A/\theta \}=\{\Delta _{A/\theta },\nabla _{A/\theta }\}$ since $s_{\theta }$ is a bounded lattice isomorphism, hence ${\rm FC}(A/\theta )={\cal B}({\rm Con}(A/\theta ))=\{\Delta _{A/\theta },\nabla _{A/\theta }\}$, therefore $\theta $ has CBLP and FCLP by Remark \ref{r4.11}.

\noindent (\ref{spec2}) Let $\theta \in {\rm Spec}(A)$, $\alpha \in {\cal B}([\theta ))$ and $\beta =\neg _{\theta }\alpha \in {\cal B}( [\theta ))$, so that $\alpha \cap \beta =\theta $ and $\alpha \vee \beta =\nabla _A$. Thus $\alpha \cap \beta \subseteq \theta $ hence $\alpha \subseteq \theta $ and $\beta \subseteq \theta $ since $\theta \in {\rm Spec}(A)$. But $\alpha ,\beta \in {\cal B}([\theta ))$, that is $\theta \subseteq \alpha $ and $\theta \subseteq \beta $. Therefore $\alpha =\theta $ or $\beta =\theta $. If $\alpha =\theta $, then $\beta =\neg _{\theta }\theta =\nabla _A$; if $\beta =\theta $, then $\alpha =\neg _{\theta }\theta =\nabla _A$. Hence ${\cal B}([\theta ))=\{\theta ,\nabla _A\}$, thus ${\cal B}({\rm Con}(A/\theta ))=\{\theta /\theta,\nabla _A/\theta \}=\{\Delta _{A/\theta },\nabla _{A/\theta }\}$ since ${\cal B}(s_{\theta })$ is a Boolean isomorphism, so ${\rm FC}(A/\theta )=\{\Delta _{A/\theta },\nabla _{A/\theta }\}$, therefore $\theta $ has CBLP and FCLP by Remark \ref{r4.11}.\end{proof}

\begin{proposition} If $A$ is local and its maximal congruence includes all its proper congruences, then $A$ has FCLP and CBLP.\label{local}\end{proposition}

\begin{proof} The result on CBLP is known from \cite{cblp}. Assume that ${\rm Max}(A)=\{\mu \}$ and ${\rm Con}(A)=(\mu ]\cup \{\nabla _A\}$. We shall prove that $A$ is FC--normal. Let $\phi ,\psi \in {\rm Con}(A)$ such that $\phi \circ \psi =\nabla _A$. Assume by absurdum that $\phi \neq \nabla _A$ and $\psi \neq \nabla _A$. Then $\phi \subseteq \mu $ and $\psi \subseteq \mu $, thus, by the transitivity of $\mu $, $\nabla _A=\phi \circ \psi \subseteq \mu \circ \mu \subseteq \mu \subsetneq \nabla _A$, so we have a contradiction. Hence $\phi =\nabla _A$ or $\psi =\nabla _A$. We may assume that $\phi =\nabla _A$, without loss of generality. Let $\alpha =\nabla _A$. Then $\alpha \in {\rm FC}(A)$ and the following hold: $\phi \vee \alpha =\nabla _A\vee \Delta _A=\nabla _A$ and $\psi \vee \neg \, \alpha =\psi \vee \neg \, \Delta _A=\psi \vee \nabla _A=\nabla _A$. Therefore $A$ is FC--normal, by Remark \ref{fccuneg}. Hence $A$ has FCLP, by Proposition \ref{fclpfcn}.\end{proof}

\begin{corollary} If $A$ is local and $\nabla 
_A$ is finitely generated, then $A$ has FCLP and CBLP.\label{corlocal}\end{corollary}

\begin{proof} The result on CBLP is known from \cite{cblp}, but also follows from the next argument. It is well known (\cite{bur}) and straightforward that, if $\nabla _A$ is finitely generated, then any proper congruence of $A$ is included in a maximal congruence of $A$. Now apply Proposition \ref{local}.\end{proof}

\begin{corollary} If $A$ is local and finite, then $A$ has FCLP and CBLP.\label{clocal}\end{corollary}

\begin{proof} If $A$ is finite, then $\nabla _A=A^2$ is finite and thus finitely generated. Now apply Corollary \ref{corlocal}.\end{proof}

\end{document}